\documentclass[12pt,oneside,a4paper,reqno]{amsart}
\usepackage[utf8]{inputenc}
\usepackage[hmargin=2.5cm,vmargin=2.5cm]{geometry}
\usepackage{microtype}

\usepackage{mathtools}
\usepackage{amssymb}
\usepackage{amsthm}
\usepackage{bbm}
\usepackage[foot]{amsaddr}
\usepackage{aligned-overset}
\usepackage{enumitem}
\usepackage{xcolor}
\usepackage{verbatim}
\usepackage[sortcites,url=false,eprint=false,isbn=false,giveninits=true,maxnames=99,minnames=99,abbreviate=true,date=year]{biblatex}
\usepackage[colorlinks,allcolors=blue]{hyperref}
\hypersetup{breaklinks=true,colorlinks,allcolors=blue,bookmarksopen,
            pdfauthor={Björn de Rijk, Joris van Winden}}
\numberwithin{equation}{section}

\newtheorem*{ttheorem}{Main results (informal summary)}
\newtheorem{theorem}{Theorem}[section]
\newtheorem{lemma}[theorem]{Lemma}

\newtheorem{hypothesis}{Hypothesis}
\newtheorem{proposition}[theorem]{Proposition}
\newtheorem{corollary}[theorem]{Corollary}
\theoremstyle{definition}

\theoremstyle{remark}
\newtheorem{remark}[theorem]{Remark}

\newcommand{\cF}{{\mathcal F}}

\newcommand{\cL}{{\mathcal L}}

\newcommand{\cO}{{\mathcal O}}

\newcommand{\bbC}{{\mathbb C}}
\newcommand{\bbN}{{\mathbb N}}

\newcommand{\bbR}{{\mathbb R}}

\newcommand{\rd}{{\,\rm{d}}}
\newcommand{\eps}{\varepsilon}

\newcommand{\Cub}{{C_{\mathrm{ub}}}}
\newcommand{\re}{{\rm e}}
\newcommand{\ri}{{\rm i}}

\newcommand{\yt}{\tilde{y}}

\DeclareMathOperator{\Real}{Re}
\DeclareMathOperator{\Imag}{Im}
\DeclareMathOperator{\supp}{supp}

\DeclarePairedDelimiter{\abs}{\lvert}{\rvert}
\DeclarePairedDelimiter{\norm}{\lVert}{\rVert}
\DeclarePairedDelimiter{\tnorm}{\lvert\kern-0.25ex\lvert\kern-0.25ex\lvert}{\rvert\kern-0.25ex\rvert\kern-0.25ex\rvert}
\DeclarePairedDelimiter{\cur}{\{}{\}}
\DeclarePairedDelimiter{\bra}{(}{)}
\DeclarePairedDelimiter{\sqr}{[}{]}

\title{Stability and dynamics of planar fronts in reaction-diffusion systems under nonlocalized perturbations}
\author{Bj\"orn de Rijk}
\address{Department of Mathematics, Karlsruhe Institute of Technology, Englerstra\ss e 2, 76131 Karlsruhe, Germany}
\email{bjoern.de-rijk@kit.edu}
\thanks{The work of BdR is funded by the Deutsche Forschungsgemeinschaft (DFG, German Research Foundation) - Project-ID 491897824 and Project-ID 258734477 - SFB 1173. JvW is supported by a DIAM fast-track scholarship.}
\author{Joris van Winden}
\address{Delft Institute of Applied Mathematics, Delft University of Technology, Mekelweg 4, 2628CD Delft, The Netherlands}
\email{J.vanWinden@tudelft.nl}

\begin{document}

\pagenumbering{arabic}
\setcounter{page}{1} 

\begin{abstract}
We analyze the stability and dynamics of bistable planar fronts in multicomponent reaction-diffusion systems on $\bbR^{d}$.
Under standard spectral stability assumptions, we establish Lyapunov stability of the front against fully nonlocalized perturbations. Such perturbations could previously be treated only for scalar equations via comparison principles.
We also prove that the leading-order dynamics of the perturbed front are governed by a modulation that tracks the motion of the front interface and evolves according to a viscous Hamilton--Jacobi equation.
This effective description reveals that asymptotic orbital stability does not hold in general. However, asymptotic stability can be recovered by imposing localization of  perturbations in the transverse spatial directions.
The treatment of nonlocalized perturbations on $\bbR^{d}$ poses significant challenges, both at the linear and nonlinear level.
At the linear level, the neutral translational mode gives rise to continuous spectrum which touches the origin and cannot be projected out by conventional means, resulting in merely algebraic decay rates for the residual. Our linear estimates are necessarily $L^{\infty}$-based, yielding significantly weaker decay rates than those available for $L^p$-localized perturbations.
At the nonlinear level, quadratic gradient terms decay at a critical rate and cannot be treated perturbatively.
We overcome these challenges by carefully decomposing the linearized dynamics,
blending semigroup methods with ideas from the stability analysis of viscous shock waves, and introducing a novel nonlinear tracking scheme that combines spatiotemporal modulation with forcing techniques and the Cole--Hopf transform.
\end{abstract}



\maketitle

\section{Introduction}
Let $d,n \in \bbN$ with $d \geq 2$. We consider multidimensional reaction-diffusion systems of the form
\begin{subequations}
\label{eq:rdefull}
\begin{align}
    \label{eq:rde}
	\partial_t u &= D\Delta u + f(u), 
	& (t,x) &\in [0,\infty) \times \bbR^{d}, \\
    \label{eq:rdeic}
    u(0,x) &= u_0(x), &  x &\in \bbR^d,
\end{align}
\end{subequations}
where $u(t,x) \in \bbR^n$, $f \colon \bbR^n \to \bbR^n$ is a smooth nonlinearity, $D \in \bbR^{n \times n}$ is a symmetric positive-definite matrix, $\Delta$ is the $d$-dimensional Laplacian, and $u_0 \colon \bbR^d \to \bbR^n$ is bounded and uniformly continuous. We are interested in the stability of planar traveling fronts against fully nonlocalized perturbations. \emph{Planar traveling fronts} are solutions to~\eqref{eq:rde} of the form
\begin{align}
    \label{eq:travelingfront}
	u_{\mathrm{tf}}(t,y,z) = \phi(z - ct),
\end{align}
with wave speed $c \in \bbR$ and smooth profile $\phi \colon \bbR \to \bbR^n$ connecting asymptotic end states $\phi_\pm = \lim_{\zeta \to \pm \infty} \phi(\zeta)$. Here we decompose $x \in \bbR^d$ as $x = (y,z)$ with $y \in \bbR^{d-1}$ and $z \in \bbR$, referring to $y$ and $z$ as the \emph{transverse} and \emph{longitudinal/horizontal} directions, respectively. We emphasize that, since we allow $\phi_+ = \phi_-$, the class of solutions under consideration includes pulses as well as genuine fronts.

Our main results may be informally summarized as follows.
\begin{ttheorem}
    Let $u_{\mathrm{tf}}$ be a solution to~\eqref{eq:rde} of the form~\eqref{eq:travelingfront}.
    Under natural spectral stability assumptions, the following statements hold:
    \begin{itemize}
        \item $u_{\mathrm{tf}}$ is Lyapunov stable against $\Cub$-perturbations, but asymptotic orbital stability does not hold.
        \item $u_{\mathrm{tf}}$ is asymptotically stable against $\Cub$-perturbations which decay, no matter how slowly, as $\abs{y} \to \infty$.
        \item Any solution $u$ to~\eqref{eq:rde} which starts near $u_{\mathrm{tf}}$ (in the supremum norm) satisfies
        \begin{equation*}
            u(t,y,z) \approx u_{\mathrm{tf}}(t,y,z - \sigma(t,y))
        \end{equation*}
        to leading order, where $\sigma$ evolves according to the viscous Hamilton--Jacobi equation
        \begin{equation}
            \label{eq:introHJ}
            \partial_t \sigma = d_{\bot}\Delta_y \sigma + \tfrac{1}{2}c\,\abs{\nabla{\sigma}}^2
        \end{equation}
        with viscosity coefficient $d_{\bot} > 0$.
    \end{itemize}
\end{ttheorem}
Precise mathematical formulations of the hypotheses and results are stated in Sections~\ref{subsec:assumptions} and~\ref{subsec:mainresults}, respectively.
The viscosity coefficient $d_{\bot}$ as well as the appropriate initial condition for~\eqref{eq:introHJ} can be computed from the adjoint zero eigenfunction of the one-dimensional linearization of~\eqref{eq:rde} about $u_{\mathrm{tf}}$; see~\eqref{eq:defdbot} and~\eqref{eq:HJ} ahead.

\subsection{Existing results}
The stability theory of traveling-front solutions to~\eqref{eq:rde} in one spatial dimension ($d=1$) is by now classical; we refer to the seminal work of Sattinger~\cite{Sattinger_stability_1976} and to the textbooks~\cite{Henry_geometric_1981,Kapitula_spectral_2013} for further references. 
In the \emph{bistable case}, where the spectra of the linearizations of~\eqref{eq:rde} about both asymptotic states $\phi_\pm$ lie in the open left-half plane, the linearization about the front exhibits a spectral gap separating the translational eigenvalue at $0$ from the remainder of the spectrum.
Exploiting the spectral gap to isolate the neutral translational mode, it follows that spectral stability implies \emph{asymptotic orbital stability}: solutions starting sufficiently close to $\phi$ converge to a spatial translate of the traveling front. 
In addition, solutions starting sufficiently close to $\phi$ remain close to the traveling front itself, i.e., it is \emph{Lyapunov stable}.
The above statements hold true for many natural function spaces, and in particular for $\Cub(\bbR)$.
As such, stability holds without any spatial localization assumptions on initial data.
This was already observed by Sattinger, cf.~\cite[\S\,8]{Sattinger_stability_1976}, and later confirmed by Fife and McLeod~\cite{Fife_approach_1977}.

In the multidimensional setting $d \geq 2$, the stability theory of traveling-front solutions to~\eqref{eq:rde} becomes significantly more delicate. The main difficulty is that invariance of the planar front in the transverse directions renders the spectrum of the linearization about $\phi$ continuous, thereby precluding a spectral gap as in one dimension. Nevertheless, after pioneering results of Jones, Levermore, and Xin~\cite{jones_asymptotic_1983,levermore_multidimensional_1992, xin_multidimensional_1992} who treated the scalar case ($n = 1$)
using comparison principles, Kapitula~\cite{kapitula_multidimensional_1997} established asymptotic stability of bistable planar fronts in multicomponent reaction-diffusion systems against localized perturbations, in the special case where $D = I$.
The analysis in~\cite{kapitula_multidimensional_1997} exploits that the leading-order dynamics are governed by the neutral translational mode in the longitudinal $z$-direction, combined with diffusive behavior in the transverse directions. In particular, these transverse diffusive effects enforce pointwise decay of $L^2$-localized data at algebraic rate $t^{-(d-1)/4}$.

This localization requirement on initial perturbations is not entirely satisfactory for several reasons.
Not only does it exclude initial perturbations corresponding to a spatial translation of the front, but it also prevents certain interesting dynamics of the perturbed front from being observed. 
For the Allen--Cahn equation, which is a special case of~\eqref{eq:rdefull} with $n = 1$, such a phenomenon is clearly illustrated in~\cite{matano_stability_2009,matano_large_2011}; see also~\cite{Roquejoffre_nontrivial_2009}. 
There it is shown that a bistable planar front is Lyapunov stable but not asymptotically orbitally stable against nonlocalized perturbations: unlike the one-dimensional case described in~\cite{Fife_approach_1977, Sattinger_stability_1976}, the front interface may undergo persistent oscillations without ever settling to a steady translate.

In the proofs in~\cite{matano_large_2011,matano_stability_2009,Roquejoffre_nontrivial_2009}, sub- and supersolutions of the form $u(t,y,z) = \phi(z - ct - \sigma(t,y))$ are constructed, where $\sigma \colon [0,\infty) \times \bbR^{d-1} \to \bbR$ satisfies the viscous Hamilton--Jacobi equation
\begin{equation*}
	\partial_t \sigma = \Delta_y \sigma + \tfrac{1}{2}c \abs{\nabla \sigma}^2,
\end{equation*}
which is exactly~\eqref{eq:introHJ} with $d_{\bot} = 1$.
In the multicomponent case ($n \geq 2$), comparison principles do not apply and there is little hope to adapt the proofs of~\cite{matano_large_2011,matano_stability_2009,Roquejoffre_nontrivial_2009}. As such, the question of whether bistable planar fronts are stable against nonlocalized perturbations when $d \geq 2$ and $n \geq 2$ has remained open until now. 
We emphasize that the analysis in~\cite{kapitula_multidimensional_1997} does not resolve this question, since stability with respect to localized perturbations does, in general, not imply stability against nonlocalized perturbations; see Remark~\ref{rem:Burgers}.
Moreover,~\cite{kapitula_multidimensional_1997} restricts to $D = I$, a special case that significantly simplifies the linear stability analysis and provides decay rates faster than those expected in the general case; see Remark~\ref{rem:Didentity} ahead.
We will spend considerable effort proving linear estimates which lift this restriction.

The contribution of this article is thus to generalize the results of~\cite{matano_stability_2009,matano_large_2011,Roquejoffre_nontrivial_2009} to general multicomponent reaction-diffusion systems, under natural spectral stability assumptions. First, we show that solutions to~\eqref{eq:rdefull} which start close to a bistable planar front $\phi$ in the supremum norm remain close and converge to a profile of the form $\phi(z - ct - \sigma(t,y))$. We obtain algebraic decay rates of $\nabla \sigma$ and $u(t,y,z) - \phi(z - ct - \sigma(t,y))$ with optimal exponents. Moreover, we show that the leading-order behavior of $\sigma$ is captured by the viscous Hamilton--Jacobi equation~\eqref{eq:introHJ}.
Using~\eqref{eq:introHJ} as an effective description of the front interface, we then recover and extend results of~\cite{matano_stability_2009} to the multicomponent case and construct an infinitely oscillating wave front, implying that asymptotic orbital stability does not hold in general. Finally, we show that asymptotic stability as in~\cite{kapitula_multidimensional_1997} can be recovered for perturbations which decay at spatial infinity (but are not necessarily $L^p$-localized for any $p < \infty$).

\subsection{Method of proof}

Our proof uses recent ideas from the nonlinear stability theory of periodic traveling-wave solutions to reaction-diffusion systems in one space dimension.
Despite the difference in the number of spatial dimensions, this setting is conceptually closely related to ours. In particular, both contexts feature diffusive spectrum touching the origin, and in both cases stability was first shown against localized perturbations. More strikingly, the modulation of a one-dimensional periodic wave and the interface motion of a two-dimensional planar front are both governed by the exact same viscous Hamilton--Jacobi equation, albeit with different coefficients.

In the periodic-wave setting, substantial effort has been devoted to gradually relaxing localization assumptions on admissible perturbations, ultimately leading to a nonlinear stability theory for fully nonlocalized $\Cub$-perturbations; see~\cite{derijk_nonlinear_2024} and references therein. The nonlinear analysis in~\cite{derijk_nonlinear_2024} inspired ours in a significant way. General similarities include the use of an inverse-modulated perturbation and the Cole--Hopf transform to handle the critical term $\abs{\nabla \sigma}^2$ in a nonperturbative manner.
Key differences are that we treat $d \geq 2$ as opposed to $d = 1$, and adopt a distinct tracking scheme. Specifically, in our nonlinear analysis the modulation $\sigma(t,y)$ of the front interface arises as a solution to the \emph{forced} viscous Hamilton--Jacobi equation
\begin{align}
\label{eq:pdesigma}
\begin{split}
\partial_t\sigma &= d_{\bot}\Delta_y \sigma + \tfrac{1}{2}c \, \abs{\nabla \sigma}^2 - g,\\
\sigma(0) &= 0,
\end{split}
\end{align}
where the forcing $g$ is chosen a posteriori so as to cancel the slowest-decaying contributions in the equation for the modulated perturbation. 
This construction avoids the quasilinear terms that appear in the $\sigma$-equation in~\cite{derijk_nonlinear_2024} and simplifies regularity control in the nonlinear iteration argument.
In addition, we note that our setting yields a geometric interpretation of the nonlinear term $\abs{\nabla \sigma}^2$, which is not available in the one-dimensional periodic case.
This interpretation is explained in Section~\ref{subsec:coeff} ahead.

Our linear stability analysis is significantly more delicate than the one in~\cite{kapitula_multidimensional_1997}, mainly for the reason that we allow diffusion matrices $D$ which are not scalar multiples of the identity.
This situation is highly relevant for applications as it is typical in reaction-diffusion models from biology, ecology, chemistry, and physics that distinct species diffuse at different rates and cross-diffusion terms occur. 
To handle general positive-definite matrices $D$, we employ an approach inspired by the pointwise Green's function analysis in~\cite{Hoff_pointwise_2002}, which was developed for the stability theory of planar shock waves in systems of viscous conservation laws~\cite{Hoff_asymptotic_2000}.
The main idea is to write the semigroup generated by the linearization in its (transverse) Fourier representation, and use Cauchy's theorem to analytically continue the frequency variable into the complex plane in a way which depends on $(y - \tilde{y}) / t$.
By distinguishing between the off-diagonal ($\abs{y - \tilde{y}} / t \gg 1$) and on-diagonal ($\abs{y - \tilde{y}} / t \lesssim 1$) regimes and carefully balancing exponential decay/growth in time, space, and frequency against each other, we are able to split off the principal part of the semigroup and obtain $L^{\infty}$-bounds on the remainder which decay at the optimal algebraic rate $t^{-1}$. In contrast to the analysis in~\cite{Hoff_pointwise_2002}, which hinges on contour deformations in the inverse Laplace representation and pointwise estimates on the resolvent kernel, our approach relies on standard semigroup theory and elementary perturbative arguments to bound the linear evolution in Fourier space. 

Altogether, we view our results as a meaningful step toward establishing multidimensional stability of planar periodic waves against $\Cub$-perturbations, as conjectured in~\cite[\S6.5]{derijk_nonlinear_2024}.

\begin{remark} \label{rem:Burgers}
To illustrate that stability of planar fronts with respect to localized perturbations does not, in general, extend to nonlocalized perturbations, let us consider the one-dimensional Burgers' equation
\begin{align} \label{eq:Burgers}
\partial_t u = \nu \partial_{zz} u - \tfrac12 \partial_z \big(u^2\big), \qquad u(t,z) \in \bbR, \, t \geq 0, \, z \in \bbR,
\end{align}
with viscosity parameter $\nu > 0$. For each $\phi_\pm \in \bbR$ with $\phi_- > \phi_+$, equation~\eqref{eq:Burgers} admits a planar traveling-front solution $u_{\mathrm{tf}}(t,z) = \Phi_{\phi_-,\phi_+}(z - ct)$. This front propagates with speed $c = (\phi_+ + \phi_-)/2$ and has the profile
\begin{align*}
\Phi_{\phi_-,\phi_+}(\zeta) = \frac{\phi_+ + \phi_-}{2} 
  + \frac{\phi_+ - \phi_-}{2} \tanh\left(\frac{\phi_- - \phi_+}{4\nu} \, \zeta\right),
\end{align*}
which connects the asymptotic states $\phi_-$ and $\phi_+$. The spatially constant states $\phi_\pm$ are only \emph{diffusively stable} as solutions to the conservation law~\eqref{eq:Burgers}. Specifically, they are asymptotically stable against $L^1$-localized perturbations with pointwise diffusive decay rate $\smash{t^{-\frac{1}{2}}}$. Moreover, the comparison principle implies that $\phi_\pm$ are Lyapunov stable against nonlocalized perturbations from $\Cub(\bbR)$. 

The front solution $u_\mathrm{tf}(t,z)$ to~\eqref{eq:Burgers} is stable with respect to localized perturbations~\cite{Ilin_behavior_1958}, and this stability persists in multiple spatial dimensions~\cite{goodman_stability_1989,Hoff_asymptotic_2000}. In contrast, the front is unstable against nonlocalized perturbations. This becomes apparent by considering the traveling-front solution $u_\varepsilon(t,z) = \Phi_{\phi_- + \varepsilon,\phi_+}(z - s t)$ to~\eqref{eq:Burgers}, which propagates with speed $s = c + \varepsilon/2$ for $\varepsilon > 0$. While 
$\|u_\varepsilon(0) - u_\mathrm{tf}(0)\|_\infty \to 0$ as $\varepsilon \downarrow 0$, we have
\begin{align*}
\limsup_{t \to \infty} \|u_\varepsilon(t) - u_\mathrm{tf}(t)\|_\infty \geq \phi_- - \phi_+
\end{align*}
for $\varepsilon > 0$, which precludes Lyapunov stability. In particular, we conclude that diffusively stable end states are insufficient to ensure stability against nonlocalized perturbations.
\end{remark}

\begin{remark}
    \label{rem:Ddiagonal}
    Our main results and hypotheses are invariant when substituting $u \mapsto A u$ in~\eqref{eq:rdefull} for any invertible matrix $A$.
    Thus, if $D$ is merely real diagonalizable, we can reduce to the case where $D$ is diagonal by such a substitution.
    Still, we emphasize that extending from the case $D = I$ (as treated in~\cite{kapitula_multidimensional_1997}) to the case of diagonal $D$ is highly nontrivial, cf.\ Remark~\ref{rem:Didentity}.
\end{remark}

\subsection{Assumptions} \label{subsec:assumptions} We formulate the hypotheses for our main results. The first hypothesis concerns the existence of a planar traveling front.
\begin{hypothesis}[Existence of planar front]
    \label{hyp:wave}
There exist a speed $c \in \bbR$ and asymptotic end states $\phi_-,\phi_+ \in \bbR^n$ such that~\eqref{eq:rde} admits a planar front solution $u_{\mathrm{tf}}(t,y,z) = \phi(z-ct)$, where the profile function $\phi \colon \bbR \to \bbR^n$ is smooth and satisfies $\lim_{\zeta \to \pm \infty} \phi(\zeta) = \phi_\pm$.
\end{hypothesis}

By passing to the co-moving frame $z \mapsto z-ct$, we find that $\phi$ is a stationary solution to
\begin{align} 	\label{eq:rde_co}
\partial_t u &= D\Delta u + c \partial_z u + f(u).
\end{align}
In the case where $d = 1$, the linearization of~\eqref{eq:rde_co} about $\phi$ is given by the closed and densely defined operator
\begin{equation}
    \label{eq:1dlinearization}
    L_0 = D \partial_{zz} + c \partial_z + f'(\phi),
\end{equation}
acting on $\Cub(\bbR)$.
From the translational invariance of~\eqref{eq:rde_co} it can be deduced that $L_0 \phi' = 0$.
The following hypothesis states that aside from this neutral translational mode, there are no other obstructions to linear asymptotic stability for $L_0$.

\begin{hypothesis}[One-dimensional spectral stability]
    \label{hyp:spectrum_1d}
    There exists $\theta_1 > 0$ such that the following conditions hold:
	\begin{enumerate}[label=(\roman*)]
        \item \label{it:spectrum_1dstable} $\Real \sigma(L_0)\setminus \cur{0} \leq -\theta_1$.
        \item \label{it:spectrum_normal} $0 \in \sigma(L_0)$, and the range of the associated spectral projection is spanned by $\phi'$.
	\end{enumerate}
\end{hypothesis}
\begin{remark}
    The spectral projection in Hypothesis~\ref{hyp:spectrum_1d}-\ref{it:spectrum_normal} is well-defined, since $0 \in \sigma(L_0)$ is isolated by Hypothesis~\ref{hyp:spectrum_1d}-\ref{it:spectrum_1dstable}.
\end{remark}

In one spatial dimension, it is well-known that Hypothesis~\ref{hyp:spectrum_1d} suffices to prove asymptotic orbital stability; see e.g.~\cite{Henry_geometric_1981,Kapitula_spectral_2013,Sattinger_stability_1976}, and references therein. It follows from a standard Sturm--Liouville argument that Hypothesis~\ref{hyp:spectrum_1d} holds for bistable fronts in the scalar case $n=1$ if and only if the front is monotone; see e.g.~\cite[Chapter~2]{Kapitula_spectral_2013}. Examples of multicomponent reaction-diffusion systems admitting bistable fronts (or pulses) satisfying Hypothesis~\ref{hyp:spectrum_1d} are the FitzHugh--Nagumo equation~\cite{cornwell_existence_2018}, the Gray--Scott system~\cite{doelman_stability_2002}, and the Gierer--Meinhardt model~\cite{doelman_large_2001}.

Hypothesis~\ref{hyp:spectrum_1d} implies that $0 \in \sigma(L_0)$ is a normal eigenvalue and therefore that $L_0$ is Fredholm (of index zero).
We note that $L_0 - \lambda$ is Fredholm for $\lambda \in \bbC$ if and only if the asymptotic operators $L_{0,\pm} - \lambda$ are invertible, where $L_{0,\pm}$ is obtained by replacing $\phi$ with $\phi_{\pm}$ in~\eqref{eq:1dlinearization}.
This in turn is equivalent to the hyperbolicity of the asymptotic matrices
\begin{align*}
    A_{\lambda,\pm} = \begin{pmatrix} 0 & I \\ D^{-1} \left(\lambda - f'(\phi_\pm)\right) & -c D^{-1}\end{pmatrix} \in \bbC^{2n \times 2n},
\end{align*}
see~\cite[Chapter~2]{Kapitula_spectral_2013} and~\cite[\S\,3]{Sandstede_stability_2002}. 
Thus, Hypothesis~\ref{hyp:spectrum_1d} can only hold if the spectra of $L_{0,\pm}$ are confined to the open left-half plane, showing that we have implicitly assumed that the front is bistable.

Additionally, it follows from Hypothesis~\ref{hyp:spectrum_1d} that zero is also a simple eigenvalue of the adjoint operator $L_0^*$, with eigenfunction $e^*$ which we may take to be normalized by
\begin{align} \label{eq:normalization}
    \langle e^*,\phi'\rangle_2 = 1.
\end{align}
We note that hyperbolicity of $A_{0,\pm}$ guarantees that both $\phi'$ and $e^*$ are exponentially localized; see~\cite[Theorem~3.2]{Sandstede_stability_2002}.

In the higher-dimensional case, the linearization of~\eqref{eq:rde_co} is instead given by the operator
\begin{equation*}
    L = D \Delta + c \partial_z + f'(\phi),
\end{equation*}
acting on $\Cub(\bbR^d)$.
Since $\phi$ is constant in the transverse direction, it is natural to apply the Fourier transform in the $y$-variable.
This motivates the introduction of the family of operators
\begin{equation} \label{eq:defLk}
    L_k = D (\partial_{zz} - k^2) + c \partial_z + f'(\phi), \qquad k \in \bbR,
\end{equation}
acting on $\Cub(\bbR)$ (note that $L_0$ is consistent with~\eqref{eq:1dlinearization}).
Indeed, letting $\cF_y$ denote the Fourier transform in the $y$-variable, we formally have
\begin{equation*}
    [\cF_y L v](\xi,z) = L_{\abs{\xi}^2} [\cF_y v](\xi,z)
\end{equation*}
for every $\xi \in \bbR^{d-1}$ and $z \in \bbR$.
This suggests that the spectrum of $L$ satisfies
\begin{align} \label{eq:spec_decomp}
\sigma(L) = \bigcup_{k \in \bbR} \sigma(L_k),
\end{align}
indicating that the operators $L_k$ are decisive for multidimensional spectral stability of the front. Although we will not prove~\eqref{eq:spec_decomp}, our analysis shows that suitable spectral assumptions on the family $\{L_k\}_{k \in \bbR}$ indeed suffice to establish stability.

From~\eqref{eq:defLk} we might naively expect that increasing $k$ will shift the spectrum of $L_k$ into the left-half plane. This intuition is correct in the high-frequency regime $k \gg 1$. Specifically, it can be proven that there exists $\theta > 0$ such that $\sigma(L_k) \leq -\theta k^2$ for $k \gg 1$.
In the low- and mid- frequency regime $k \lesssim 1$ however, our assumptions so far do not guarantee spectral stability of $L_k$, and there are two potential obstructions.

The first possible obstruction occurs in the low-frequency regime.
Here, the eigenvalue $0 \in \sigma(L_0)$ can be analytically continued for $k \ll 1$, resulting in a critical branch of spectrum $\lambda_{\mathrm{lin}}(k) \in \sigma(L_k)$ which touches the origin and is commonly referred to as the \emph{linear dispersion relation}. By a Lyapunov--Schmidt reduction argument, we obtain the expansion
\begin{align}
    \label{eq:lambdac}
    \lambda_{\mathrm{lin}}(k) = -d_{\bot}k^2 + \cO(k^4),
\end{align}
where the coefficient $d_{\bot}$ is given by the integral
\begin{align}
	\label{eq:defdbot}
	d_{\bot} \coloneq \langle e^*,D \phi'\rangle_2 = \int_{\bbR} \langle e^*(z), D \phi'(z)\rangle \rd z.
\end{align}
Depending on the sign on $d_{\bot}$, the critical branch $\lambda_{\mathrm{lin}}(k)$ might lie in the left-half plane, right-half plane, or be degenerate.
The long-wavelength instability in the second case, $d_\bot < 0$, is often termed \emph{transverse sideband instability}.
Even with the assumption introduced below, which ensures that the first scenario, $d_\bot > 0$, occurs, the presence of marginal continuous spectrum poses a significant challenge:
unlike in the one-dimensional case, a projection onto the translational eigenmode does not open a spectral gap.
This phenomenon manifests itself in our linear stability analysis, where we only obtain algebraic decay with rate $t^{-1}$, cf.~Theorem~\ref{thm:decomposition}, which is expected to be optimal under the present level of generality.

The second possible obstruction occurs in the mid-frequency regime $k \sim 1$, where the spectrum of $L_k$ might cross the imaginary axis in a \emph{short-wave transverse instability}.
To rule out this possibility, we impose that the spectrum of $L_k$ is contained in the open left-half plane when $k$ is bounded away from zero.

The above discussion motivates the following natural transverse spectral stability assumption.

\begin{hypothesis}[Transverse spectral stability]
    \label{hyp:spectrum_multid}
    There exist constants $k_0,\theta_2,\theta_3 > 0$ such that the following hold:
	\begin{enumerate}[label=(\roman*)]
		\item \label{it:spectrum_lowfreq} $\Real \sigma(L_k) \leq -\theta_2 k^2$ for all $\abs{k} \leq k_0$.
        \item \label{it:spectrum_midhighfreq} $\Real \sigma(L_k) \leq -\theta_3$ for all $\abs{k} \geq k_0$.
	\end{enumerate}
\end{hypothesis}

As explained in Remark~\ref{rem:Didentity}, Hypothesis~\ref{hyp:spectrum_multid} follows from Hypothesis~\ref{hyp:spectrum_1d} if the diffusion matrix $D$ is a multiple of the identity. We refer to~\cite{cornwell_existence_2018} for an example of a two-component reaction-diffusion system with $D = I$ admitting spectrally stable one-dimensional traveling pulses satisfying Hypothesis~\ref{hyp:spectrum_1d}, whose multidimensional counterparts thus satisfy Hypothesis~\ref{hyp:spectrum_multid}. In general, however, Hypothesis~\ref{hyp:spectrum_multid} may fail even when the underlying one-dimensional front satisfies Hypothesis~\ref{hyp:spectrum_1d}.
Bistable reaction-diffusion systems admitting such \emph{transversely unstable} fronts include combustion models~\cite{terman_stability_1990} and activator-inhibitor systems~\cite{taniguchi_instability_2003}; we refer to~\cite{carter_criteria_2023} for further examples.

Assuming Hypothesis~\ref{hyp:spectrum_1d}, it is clear from~\eqref{eq:lambdac} that Hypothesis~\ref{hyp:spectrum_multid}-\ref{it:spectrum_lowfreq} is satisfied if and only if $d_{\bot} > 0$, which can be verified on a case-by-case basis by computing the sign of the integral~\eqref{eq:defdbot}. On the other hand, outside of certain degenerate scenarios such as the case $D = I$, there is no general criterion that ensures Hypothesis~\ref{hyp:spectrum_multid}-\ref{it:spectrum_midhighfreq} and its verification is typically nontrivial; see, for instance,~\cite{bastiaansen_stable_2019}. 

\subsection{Main results}
\label{subsec:mainresults}
From this point onward, Hypotheses~\ref{hyp:wave}, \ref{hyp:spectrum_1d}, and~\ref{hyp:spectrum_multid} will be in force throughout. In particular, all objects appearing in these hypotheses (most notably the profile $\phi$) will be regarded as fixed.

Our first main result states that any solution to~\eqref{eq:rde} which starts sufficiently close to the planar front converges to a smoothly modulated version of the front.

\begin{theorem}[Convergence to modulated front]
	\label{thm:main}
	There exist constants $C,\eps > 0$ such that, whenever $u_0 \in \Cub(\bbR^d)$ satisfies
    \begin{align} \label{eq:E0ineq}
    E_0 \coloneq \norm{u_0 - \phi}_{\infty} \leq \eps,
   \end{align}
   the following statements hold:
   \begin{itemize}
        \item There exists a unique global solution $u$ to~\eqref{eq:rdefull} with regularity
        \begin{align} \label{eq:soluglobal}
        u \in C\big([0,\infty),\Cub(\bbR^d)\big) \cap C^\infty\big((0,\infty) \times \bbR^d\big).
        \end{align}
        \item There exists a smooth modulation function $\sigma \colon [0,\infty) \times \bbR^{d-1} \to \bbR$ such that the estimates
        \begin{align}
            \label{eq:thmest4}
    		\sup_{(y,z) \in \bbR^d} \abs{u(t,y,z) - \phi(z - ct - \sigma(t,y))} &\leq C E_0 \frac{\log(2+t)}{1+t}
        \end{align}
        and
        \begin{align}       
            \label{eq:thmest1}
    		\norm{\sigma(t)}_\infty + \sqrt{1+t}\, \norm{\nabla \sigma(t)}_\infty + (1+t)\, \norm{\Delta \sigma(t)}_\infty &\leq CE_0
        \end{align}
        hold for all $t \geq 0$.
    \end{itemize}
\end{theorem}

Notably, the temporal decay rates presented in Theorem~\ref{thm:main} are sharp (up to a possibly a logarithm); we refer to Section~\ref{subsec:expect_decay} for details. An immediate consequence of Theorem~\ref{thm:main} is that the planar front is stable in the classical Lyapunov sense.

\begin{corollary}[Lyapunov stability]
	\label{cor:lyapunov}
    There exist constants $C,\eps > 0$ such that for every $u_0 \in \Cub(\bbR^d)$ satisfying~\eqref{eq:E0ineq} the following statements hold:
    \begin{itemize}
        \item There exists a unique global solution $u$ to~\eqref{eq:rdefull} with regularity~\eqref{eq:soluglobal}.
        \item The solution $u$ satisfies
        \begin{equation}
        \label{eq:lyapunovstable}
    		\norm{u(t) - u_{\mathrm{tf}}(t)}_{\infty} \leq C E_0
    	\end{equation}
        for all $t\geq 0$, where $u_{\mathrm{tf}}$ is the planar front~\eqref{eq:travelingfront}.
    \end{itemize}
\end{corollary}
Our second main result gives a more precise description of the leading-order dynamics of the perturbed planar front. Specifically, it shows that the modulational dynamics of the front interface are governed by the viscous Hamilton--Jacobi equation~\eqref{eq:introHJ} which we state here again, now supplemented with the appropriate initial condition:
\begin{subequations}
\label{eq:HJ}
\begin{align}
    \label{eq:HJ_pde}
    \partial_t \tilde{\sigma} &= d_{\bot}\Delta_y \tilde{\sigma} + \tfrac{1}{2}c\,\abs{\nabla{\tilde{\sigma}}}^2, \\
    \label{eq:HJ_ic}
    \tilde{\sigma}(0,y) &= \langle e^*, \phi - u_0(y,\cdot) \rangle_2.
\end{align}
\end{subequations}
We recall from Section~\ref{subsec:assumptions} that the viscosity coefficient $d_\bot$ in~\eqref{eq:HJ} is related to the second derivative of the linear dispersion relation $\lambda_{\mathrm{lin}}(k)$ and is given by~\eqref{eq:defdbot}. The transverse long-wavelength instability arising for $d_\bot<0$ thus corresponds to the ill-posedness of the modulation equation~\eqref{eq:HJ}. Hypothesis~\ref{hyp:spectrum_multid} ensures, however, that in our case we have $d_\bot>0$. The coefficient $\frac12 c$ of the nonlinearity in~\eqref{eq:HJ} can be interpreted via a nonlinear dispersion relation, which connects the wave speed to the orientation of the planar front; see Section~\ref{subsec:coeff} for further details. 

\begin{theorem}[Effective front dynamics]
	\label{thm:HJ}
    There exist constants $C,\eps > 0$ such that, whenever $u_0 \in \Cub(\bbR^d)$ satisfies~\eqref{eq:E0ineq}, the following statements hold:
    \begin{itemize}
        \item  There exists a unique global solution $u$ to~\eqref{eq:rdefull} with regularity~\eqref{eq:soluglobal}.
        \item There exists a unique global solution $\tilde{\sigma}$ to~\eqref{eq:HJ} with regularity
        \begin{equation}
            \label{eq:sigmatildereg}
            \tilde\sigma \in C\big([0,\infty),\Cub(\bbR^{d-1})\big) \cap C^\infty\big((0,\infty) \times \bbR^{d-1}\big).
        \end{equation}
        \item The solutions $u$ and $\tilde{\sigma}$ satisfy the estimate
        \begin{equation}
            \label{eq:HJeffest}
    		\abs{u(t,y,z) - \phi(z - ct - \tilde{\sigma}(t,y))} \leq C\big(E_0^2 + E_0(1+t)^{-1}\log(2+t)\big)
    	\end{equation}
        for all $t\geq 0$ and $(y,z) \in \bbR^d$.
    \end{itemize}
\end{theorem}

We emphasize that neither Theorem~\ref{thm:main} nor Theorem~\ref{thm:HJ} implies asymptotic (orbital) stability in the sense that $u(t)$ converges to (any translate of) the planar front.  
Unlike the case of $L^2$-localized perturbations considered in~\cite{kapitula_multidimensional_1997}, asymptotic stability fails for general $\Cub$-perturbations.  
For the specific case of the Allen--Cahn equation, this was demonstrated in~\cite[Proposition~1.9]{matano_stability_2009}.  
Using the effective description of the front dynamics provided by Theorem~\ref{thm:HJ}, we show that a similar counterexample can be constructed for planar fronts in general multicomponent reaction-diffusion systems.
\begin{corollary}[No asymptotic orbital stability]
    \label{cor:oscillating}
    For every $\eps > 0$, there exists an initial condition $u_0 \in \Cub(\bbR^d)$ satisfying~\eqref{eq:E0ineq} such that the following statements hold:
    \begin{itemize}
        \item There exists a unique global solution $u$ to~\eqref{eq:rdefull} with regularity~\eqref{eq:soluglobal}.
        \item The solution $u$ satisfies
        \begin{equation}
            \label{eq:infosc}
            \limsup_{t \to \infty} \, \inf_{a \in \bbR}
        \, \sup_{(y,z) \in \bbR^d} \, \abs{u(t,y,z) - \phi(z - a)} > 0.
        \end{equation}
    \end{itemize}
\end{corollary}

Our final result shows that asymptotic stability can be recovered when the initial perturbation is localized in the transverse spatial directions. This localization condition is significantly weaker than the $L^2$-integrability requirement in all spatial directions imposed in~\cite{kapitula_multidimensional_1997}. In particular, it requires neither localization in the longitudinal direction nor integrability in the transverse directions. As expected, this weaker assumption does not yield an explicit convergence rate.
\begin{theorem}[Transverse spatial localization yields asymptotic stability]
    \label{thm:decay}
    There exists a constant $\eps > 0$ such that, whenever $u_0 \in \Cub(\bbR^d)$ satisfies~\eqref{eq:E0ineq} and
	\begin{equation*}
		\lim_{\abs{y} \to \infty} \, \sup_{z \in \bbR} \, \abs{u_0(y,z) - \phi(z)} = 0,
	\end{equation*}
    there exists a unique global solution $u$ to~\eqref{eq:rdefull} with regularity~\eqref{eq:soluglobal}, which  satisfies
    \begin{equation*}
        \lim_{t \to \infty}\norm{u(t) - u_{\mathrm{tf}}(t)}_{\infty} = 0.
    \end{equation*}
\end{theorem}

\subsection{Geometric interpretation of nonlinear coefficient}
\label{subsec:coeff}
We provide a geometric interpretation of the coefficient $\frac12 c$ of the nonlinearity in the viscous Hamilton--Jacobi equation~\eqref{eq:HJ}. To this end, we first derive the following identity via Lyapunov--Schmidt reduction, which expresses this coefficient as a pairing with the adjoint zero eigenfunction.
\begin{lemma} \label{lem:ndispersion}
It holds
\begin{align}
    \label{eq:nonlineardispersion}
     \tfrac{1}{2}c = -\langle e^*,D\phi''\rangle_2 = -\int_{\bbR} \langle e^*(z),D\phi''(z)\rangle \rd z.
\end{align}
\end{lemma}
\begin{proof}
Observe that for every $\alpha > 0$, the rescaled profile $\phi_{\alpha}(z) \coloneq \phi(\alpha^{-1}z)$ satisfies the following rescaled traveling-wave equation:
\begin{equation}
    \label{eq:rescaledode}
	\alpha^2 D \phi_{\alpha}'' + {\alpha}c \phi_a' + f(\phi_{\alpha}) = 0.
\end{equation}
Differentiating~\eqref{eq:rescaledode} with respect to $\alpha$ at $\alpha = 1$ and writing $\psi$ for the resulting derivative, we find that $\psi$ satisfies
\begin{equation*}
	L_0\psi = -2 D \phi'' - c \phi'.
\end{equation*}
Taking the inner product with the adjoint eigenfunction $e^*$ and using~\eqref{eq:normalization} and $\langle e^*,L_0 \psi\rangle_2 = \langle L_0^* e^*,\psi\rangle_2 = 0$,~\eqref{eq:nonlineardispersion} follows.
\end{proof}

Although the proof above is concise and the identity~\eqref{eq:nonlineardispersion} will be useful to eliminate secular terms in the nonlinear iteration, it does not clarify the role of the rescaled profile $\phi_{\alpha}$ in interpreting the coefficient of the nonlinearity in~\eqref{eq:HJ}. The following argument provides a more geometric perspective. For the sake of exposition, we restrict to $d = 2$.
By rotational invariance of~\eqref{eq:rde}, the profile $\phi$ generates a family of rotated planar-front solutions
\begin{align*}
	u_{\gamma}(t,y,z) = \phi\bra[\bigg]{\frac{z - \gamma y}{\sqrt{1+\gamma ^2}} - ct}, \qquad \gamma \in \bbR.
\end{align*}
Introducing $\alpha \coloneq \smash{\sqrt{1+\gamma^2}}$ and $\sigma(t,y) \coloneq \smash{(\sqrt{1+\gamma^2}-1)} c t + \gamma y$, we obtain the identities
\begin{align*}
	u_{\gamma}(t,y,z) = \phi_{\alpha}(z - \gamma y - \sqrt{1+\gamma^2} \, c t) =  \phi_{\alpha}(z - ct - \sigma(t,y)).
\end{align*}
The first identity suggests interpreting  $u_{\gamma}$ as an oblique front with profile $\phi_{\alpha}$, traveling horizontally with speed $\smash{\sqrt{1+\gamma^2} \, c}$, which explains the use of the rescaled profile in the proof of~\eqref{eq:nonlineardispersion}.
The second identity, which is more in line with our upcoming analysis, instead suggests viewing the solution as a straight front moving at horizontal speed $c$, modulated by $\sigma(t,y)$. From this perspective, the modulation satisfies
\begin{align*}
    \partial_t\sigma = \sqrt{1+\gamma^2}\, c - c = \tfrac{1}{2}c\gamma^2 + \cO(\gamma^3),
    \qquad \Delta \sigma = 0, \qquad \abs{\nabla \sigma}^2 = \gamma^2,
\end{align*}
showing that~\eqref{eq:HJ_pde} indeed governs $\sigma$ to leading order in $\gamma$.
This perspective also reveals the connection to the nonlinear dispersion relation $c(\gamma) = \smash{\sqrt{1+\gamma^2}}\, c \approx c + \tfrac{1}{2}\gamma^2c $, which connects the horizontal speed of $u_{\gamma}$ to the orientation of the front.

\subsection{Notation} \label{subsec:notation}
Let $\ell,m \in \bbN$ and $\mathbb{F} \in \{\bbR,\bbC\}$. On $\mathbb{F}^\ell$ we denote by $|\cdot|$ the standard Euclidean norm and by $\langle \cdot,\cdot \rangle$ the associated inner product. Moreover, we let $|x|_\infty = \max\{|x_1|,\ldots,|x_\ell|\}$ for $x \in \mathbb{F}^\ell$.

For a Banach space $X$, we denote by $\mathcal{B}(X)$ the Banach space of bounded operators on $X$. We write $\Cub(\bbR^m;\mathbb{F}^\ell)$ for the Banach space of bounded uniformly continuous functions from $\bbR^m$ to $\mathbb{F}^\ell$, endowed with the supremum norm $\|f\|_\infty = \sup\{|f(x)| : x \in \bbR^m\}$. When the codomain is clear from context or not essential, we simply write $\Cub(\bbR^m)$ instead of $\Cub(\bbR^m;\mathbb{F}^\ell)$. We use $\tnorm{T}$ to abbreviate the operator norm $\|T\|_{\mathcal{B}(\Cub(\bbR^m))}$.

For $w \in L^1(\bbR)$ and $v \in \Cub(\bbR)$ we set
\begin{align*}
    \langle w, v\rangle_2 \coloneq \int_{\bbR}\langle w(z),v(z)\rangle \rd z
\end{align*}
to denote the $L^2$-pairing. If $w \in L^1(\bbR)$ and $v \in \Cub(\bbR^d)$, we define
\begin{align*}
\langle w, v\rangle_2 \coloneq y \mapsto \int_{\bbR}\langle w(z),v(y,z)\rangle \rd z,
\end{align*}
which yields a function in $\Cub(\bbR^{d-1})$.

Throughout the article, we fix a smooth nonnegative cutoff function $\varrho \colon \bbR \to \bbR$ which satisfies $\int_{\bbR}\varrho(s)\rd s = 1$ and $\supp \varrho \subset [1/4,3/4]$. We also fix a smooth nondecreasing $\chi \colon \bbR \to \bbR$ which satisfies $\chi(s) = 0$ for $s \leq 1/8$ and $\chi(s) = 1$ for $s \geq 1/4$.

Finally, let $S$ be a set, and let $A, B \colon S \to \bbR$. Throughout the paper, we abbreviate the expression ``there exists a constant $C>0$ such that $A(x) \leq CB(x)$ for all $x \in S$'' by adopting the notation ``$A(x) \lesssim B(x)$ for all $x \in S$''.

\section{Linear theory: analysis of Fourier symbol}
\label{sec:spectral}

In this section, we study the transverse Fourier symbol of the linearization $L$ of the reaction-diffusion system~\eqref{eq:rde_co} about the planar front $\phi$. The Fourier symbol 
\begin{align*} 
\hat{L}_\xi = D \left(\partial_{zz} - \lambda_0(\xi)\right) + c \partial_z + f'(\phi), \qquad \xi \in \bbC^{d-1},
\end{align*}
is a closed and densely defined operator on $\Cub(\bbR)$, where $\lambda_0 \colon \bbC^{d-1} \to \bbC$ denotes the analytic function
\begin{align*}
\lambda_0(\xi) = \sum_{i = 1}^{d-1} \xi_i^2. 
\end{align*}
We allow the Fourier frequency variable $\xi$ to take complex values so that we can later access results from analytic function theory; see Section~\ref{subsec:Fourier_rep} for further details.

We begin by recording several direct consequences of the spectral stability assumptions in Hypotheses~\ref{hyp:spectrum_1d} and~\ref{hyp:spectrum_multid} for the spectrum of $\hat{L}_\xi$ at low frequencies $|\xi| \ll 1$. We then establish bounds on the semigroup $\smash{\re^{t\hat{L}_\xi}}$ in low-, mid-, high-, and all-frequency regimes, which will be used in Section~\ref{sec:linear} to analyze the linearized dynamics of~\eqref{eq:rde_co}.

\subsection{Low-frequency spectral analysis}

Since $\smash{\hat{L}_\xi}$ depends analytically on $\smash{(\xi_1^2,\ldots,\xi_{d-1}^2)}$ and $0$ is an isolated eigenvalue of $\smash{\hat{L}}_0 = L_0$ whose spectral projection has rank $1$ by Hypothesis~\ref{hyp:spectrum_1d}, it can be extended analytically in $\smash{(\xi_1^2,\ldots,\xi_{d-1}^2)}$ by standard perturbation theory, cf.~\cite[\S\,II.1.8 and \S\,VII.1.3]{Kato_Perturbation_1995}.
This yields an analytic function $\lambda_{\mathrm{c}}(\xi)$ of isolated eigenvalues of $\smash{\hat{L}_\xi}$ for all $|\xi|\ll1$ with $\lambda_{\mathrm{c}}(0) = 0$. The corresponding spectral projections $P_{\xi}$ have rank $1$ and also depend analytically on $\smash{(\xi_1^2,\ldots,\xi_{d-1}^2)}$. Because $\lambda_{\mathrm{c}}(\xi)$ is a simple eigenvalue of $\smash{\hat{L}_\xi}$, its conjugate $\smash{\overline{\lambda_{\mathrm{c}}(\xi)}}$ is a simple eigenvalue of the adjoint $\smash{\hat{L}_\xi^*}$. The associated eigenfunctions, as well as their derivatives, are exponentially localized by~\cite[Theorem~3.2]{Sandstede_stability_2002}. Finally, a standard computation using Lyapunov--Schmidt reduction yields the leading-order coefficients in the Taylor expansion of $\lambda_{\mathrm{c}}(\xi)$ near $\xi = 0$. These facts are summarized in the following proposition.

\begin{proposition}[Low-frequency spectrum] \label{prop:speccons}
There exist an open neighborhood $U \subset \bbC^{d-1}$ of $0$, a constant $C_0 > 0$, and holomorphic functions $\lambda_{\mathrm{c}} \colon U \to \bbC$ and $P \colon U \to \mathcal{B}(\Cub(\bbR))$ such that the following statements hold for every $\xi \in U$:
\begin{enumerate}
    \item $\Real \sigma(\hat{L}_\xi) \setminus \{\lambda_{\mathrm{c}}(\xi)\} < -\theta_1 / 2$, where $\theta_1$ is as in Hypothesis~\ref{hyp:spectrum_1d}.
    \item $\lambda_{\mathrm{c}}(\xi)$ is an isolated eigenvalue of $\hat{L}_{\xi}$, and the associated spectral projection $P_{\xi}$ has rank $1$.
    \item With $d_{\bot}$ defined as in~\eqref{eq:defdbot}, it holds that $d_{\bot} > 0$, as well as:
    \begin{align*}
\abs{\lambda_{\mathrm{c}}(\xi) + d_\bot \lambda_0(\xi)} \leq C_0 |\xi|^4, \qquad \norm{P_{\xi} - P_0}_{\mathcal{B}(\Cub(\bbR))} \leq C_0 |\xi|^2.
\end{align*}
\end{enumerate}
Finally, the spectral projection $P_0 \in \mathcal{B}(\Cub(\bbR))$ associated with $0 \in \sigma(L_0)$ is given by
\begin{align} \label{eq:defP0}
P_0 v = \langle e^*,v\rangle_2 \, \phi',
\end{align}
where the adjoint eigenfunction $e^* \in \ker(L_0^*)$, satisfying~\eqref{eq:normalization}, was introduced in Section~\ref{subsec:assumptions}. The eigenfunctions $\phi'$ and $e^*$, as well as their derivatives, are smooth and exponentially localized.
\end{proposition}

\subsection{Semigroup bounds}
\label{subsec:semigroupbounds}

Since $\hat{L}_\xi$ is a sectorial operator on $\Cub(\bbR)$, cf.~\cite[Corollary~3.1.9]{Lunardi_analytic_2013}, it generates an analytic semigroup $\smash{\re^{t\hat{L}_\xi}}$. In the following, we establish estimates on $\smash{\re^{t\hat{L}_\xi}}$, which will be used in the upcoming linear stability analysis in Section~\ref{sec:linear}. 

We start with the following all-frequency estimates.

\begin{lemma}[All-frequency bounds]
\label{lem:semigroup_allfreq}
There exist $\mu_1,\kappa_1 >0$ such that
\begin{align}
    \label{eq:semigroup_allfreq}
    \tnorm{\re^{t\hat{L}_{\xi_0 + \ri \xi_1}}} + \abs{\re^{-td_{\bot}\lambda_0(\xi_0 + \ri \xi_1)}} \lesssim \re^{\kappa_1 t - \mu_1 \abs{\xi_0}^2 t + \kappa_1\abs{\xi_1}^2t}
\end{align}
for all $\xi_0,\xi_1 \in \bbR^{d-1}$ and $t \geq 0$.
\end{lemma}
\begin{proof}
    The estimate on $\re^{-t d_{\bot} \lambda_0(\xi_0 + \ri \xi_1)}$ is elementary, so we only estimate $\re^{t\hat{L}_{\xi_0 + \ri \xi_1}}$.
    We first consider the case where $\xi_1 = 0$.
    We introduce the operators 
    \begin{equation*}
        \cL_{\xi} = D \partial_{zz} + c\partial_z - \abs{\xi}^2 D
    \end{equation*}
    for $\xi \in \bbR^{d-1}$. Since $D$ is symmetric and positive-definite, there exists an invertible matrix $J \in \bbR^{n \times n}$ and a positive diagonal matrix $D_0 \in \bbR^{n \times n}$ such that $D = J D_0 J^{-1}$. Hence, using the classical bound 
    \begin{align*} \tnorm{\re^{t \left(\partial_{zz} + c \partial_z\right)}} \leq 1\end{align*} 
    on the scalar convective heat semigroup for $t \geq 0$, see e.g.~\cite[Proposition~3.6]{derijk_nonlinear_2024}, we find $\mu_0 > 0$ such that
    \begin{align*} 
        \tnorm{\re^{t\cL_{\xi_0}}} 
        = \tnorm{J\re^{t(D_0 \partial_{zz} + c \partial_z - \abs{\xi_0}^2 D_0)}J^{-1}}
        \lesssim \tnorm{\re^{-t \abs{\xi_0}^2 D_0}}
        \lesssim \re^{-\mu_0 \abs{\xi_0}^2 t}
    \end{align*}
    for all $\xi_0 \in \bbR^{d-1}$ and $t \geq 0$.
    Since $\hat{L}_{\xi_0} = \cL_{\xi_0} + f'(\phi)$ and $f'(\phi)$ is a bounded operator, it follows from the bounded perturbation theorem, cf.~\cite[Theorem III.1.3]{engel_oneparameter_2000}, that there exists $C > 0$ such that we have
    \begin{align}
        \label{eq:semigroup_allfreqreal}
        \tnorm{\re^{t\hat{L}_{\xi_0}}}\lesssim \re^{Ct -\mu_0 \abs{\xi_0}^2t}
    \end{align}
    for all $\xi_0 \in \bbR^{d-1}$ and $t \geq 0$.
    To extend to the case where $\xi_1 \neq 0$, we first observe that, by Young's inequality, we can estimate
    \begin{align}
        \label{eq:xi1absorbest}
        \tnorm{\hat{L}_{\xi_0 + \ri \xi_1} - \hat{L}_{\xi_0}} 
        \lesssim \abs{\xi_0}\abs{\xi_1} + \abs{\xi_1}^2 \lesssim \delta \abs{\xi_0}^2 + (1+\delta^{-1})\abs{\xi_1}^2
    \end{align}
    for all $\xi_0,\xi_1 \in \bbR^{d-1}$ and $\delta > 0$.
    Thus, applying the bounded perturbation theorem again and choosing $\delta > 0$ sufficiently small, we find a constant $\kappa > 0$ such that
    \begin{align*}
        \tnorm{\re^{t\hat{L}_{\xi_0 + \ri \xi_1}}} \lesssim \re^{Ct - \frac{\mu_0}{2}\abs{\xi_0}^2t +  \kappa \abs{\xi_1}^2}
    \end{align*}
    for all $\xi_0,\xi_1 \in \bbR^{d-1}$ and $t \geq 0$.
    Thus,~\eqref{eq:semigroup_allfreq} holds with $\mu_1 = \mu_0 / 2$ and $\kappa_1 = \max\cur{C,\kappa}$.
\end{proof}

In the next two lemmas, we establish low- and mid-high-frequency estimates on the semigroup $\smash{\re^{t\hat{L}_\xi}}$.

\begin{lemma}[Low-frequency bounds]
\label{lem:semigroup_lowfreq}
There exist $k, \mu_2,\kappa_2 >0$ such that we have
\begin{align}
    \label{eq:semigroup_projected}
    \tnorm{\re^{t\hat{L}_{\xi_0 + \ri \xi_1}}(I - P_{\xi_0 + \ri \xi_1})} 
    &\lesssim \re^{-\mu_2 t} \\
    \label{eq:semigroup_taylor}
    \tnorm{\re^{t\hat{L}_{\xi_0 + \ri \xi_1}}P_{\xi_0 + \ri\xi_1} - \re^{-td_{\bot}\lambda_0(\xi_0 + \ri \xi_1)}P_0}
    &\lesssim t^{-1} \re^{-\mu_2\abs{\xi_0}^2 t + \kappa_2 \abs{\xi_1}^2 t}
\end{align}
for all $t \geq 1$ and $\xi_0,\xi_1 \in \bbR^{d-1}$ satisfying $\max\cur{\abs{\xi_0},\abs{\xi_1}} \leq k$.
\end{lemma}
\begin{proof}
    Let $U \subset \bbC^{d-1}$ be as in Proposition~\ref{prop:speccons} and choose $k > 0$ such that $V \coloneq \cur{\xi_0 + \ri \xi_1 : \max\cur{\abs{\xi_0},\abs{\xi_1}} \leq k} \subset U$.
    Since $\smash{\re^{t\hat{L}_{\xi}}}$ is an analytic semigroup for every $\xi \in \bbC^{d-1}$, its spectral bound equals its growth bound by~\cite[Corollary 3.12]{engel_oneparameter_2000}.
    Thus, it follows from Proposition~\ref{prop:speccons} that for all $\xi \in V$ there exists $M_\xi > 0$ such that
    \begin{align*}
        \tnorm{\re^{t\hat{L}_{\xi}}(I - P_{\xi})} \leq M_{\xi} \re^{-\frac{\theta_1}{4} t}.
    \end{align*}
    Hence, applying Lemma~\ref{lem:uniformsemigroup}, we obtain the uniform bound~\eqref{eq:semigroup_projected} at the cost of replacing $\theta_1/4$ by $\mu_2 = \theta_1/8$.
    
    For the second estimate, we first rewrite
    \begin{align*}
        \re^{t\hat{L}_{\xi}}P_{\xi} - \re^{-td_{\bot}\lambda_0(\xi)}P_0 
        &= \re^{t\lambda_{\mathrm{c}}(\xi)}P_{\xi} - \re^{-t d_{\bot}\lambda_0(\xi)}P_0\\ 
        &= \re^{-t d_{\bot}\lambda_0(\xi)} ((\re^{t\lambda_{\mathrm{c}}(\xi) + t d_{\bot} \lambda_0(\xi)} - 1)P_{\xi} + (P_{\xi} - P_0))
    \end{align*}
    for $\xi \in V$ and $t \geq 1$.
    Thus, using the inequality $\abs{\re^{w} - 1} \leq \re^{\abs{w}} - 1 \leq \abs{w}\re^{\abs{w}}$ for $w \in \bbC$ and invoking the Taylor expansions from Proposition~\ref{prop:speccons}, we find
    \begin{align*}
        \tnorm{\re^{t\hat{L}_{\xi}}P_{\xi} - \re^{-td_{\bot}\lambda_0(\xi)}P_0} 
        &\lesssim (\abs{\xi}^2 + t\abs{\xi}^4) \re^{- d_{\bot}\abs{\Real(\xi)}^2 t +  d_{\bot}\abs{\Imag(\xi)}^2 t + C_0\abs{\xi}^4 t}
    \end{align*}
    for all $\xi \in V$ and $t \geq 1$.     
    By additionally imposing $k \leq \sqrt{d_{\bot}/(4C_0)}$, we can bound
    \begin{align*}
        \tnorm{\re^{t\hat{L}_{\xi}}P_{\xi} - \re^{-td_{\bot}\lambda_0(\xi)}P_0}
        &\lesssim t^{-1} (t\abs{\xi}^2 + t^2\abs{\xi}^4) \re^{- \frac{1}{2} d_{\bot}\abs{\Real(\xi)}^2 t +  
        \frac{3}{2}d_{\bot}\abs{\Imag(\xi)}^2 t} \\
        &\lesssim t^{-1} \re^{-\frac14 d_{\bot}\abs{\Real(\xi)}^2 t + 2d_{\bot} \abs{\Imag(\xi)}^2t}
    \end{align*}
  for all $\xi \in V$ and $t \geq 1$, which shows that~\eqref{eq:semigroup_taylor} holds with $\mu_2 = d_{\bot}/4$ and $\kappa_2 = 2 d_{\bot}$.
\end{proof}

\begin{lemma}[Mid- and high-frequency bounds]
\label{lem:semigroup_midhighfreq}
Let $k > 0$ be as in Lemma~\ref{lem:semigroup_lowfreq}. 
There exist $\mu_3,\kappa_3 > 0$ such that
\begin{align}
    \label{eq:semigroup_midhighfreq}
    \tnorm{\re^{t\hat{L}_{\xi_0 + \ri \xi_1}}} + \abs{\re^{-td_{\bot}\lambda_0(\xi_0 + \ri \xi_1)}} \lesssim \re^{-\mu_3 t  - \mu_3\abs{\xi_0}^2 t + \kappa_3\abs{\xi_1}^2 t}
\end{align}
for all $t \geq 0$ and $\xi_0,\xi_1 \in \bbR^{d-1}$ with $\abs{\xi_0} \geq k$.
\end{lemma}
\begin{proof}
    The estimate for $\re^{-td_{\bot}\lambda_0(\xi)}$ is elementary, so we only treat $\re^{t\hat{L}_{\xi}}$.
    We first consider the case $\xi_1 = 0$.
    Letting $\mu_0, C > 0$ be as in~\eqref{eq:semigroup_allfreqreal}, we introduce $k_1 = \smash{\sqrt{2 C \mu_0^{-1} + 1}}$. 
    It then follows from~\eqref{eq:semigroup_allfreqreal} that we have
    \begin{align}  \label{eq:highfreq}
        \tnorm{\re^{t\hat{L}_{\xi_0}}} \lesssim \re^{-\frac{\mu_0}{2} t - \frac{\mu_0}{2}\abs{\xi_0}^2 t}
    \end{align}
    for all $t \geq 0$ and $\xi_0 \in \bbR^{d-1}$ with $\abs{\xi_0} \geq k_1$.    
    
    On the other hand, since the spectral bound of the analytic semigroup $\smash{\re^{t\hat{L}_\xi}}$ equals its growth bound by~\cite[Corollary 3.12]{engel_oneparameter_2000}, we deduce with the aid of Hypothesis~\ref{hyp:spectrum_multid} and Lemma~\ref{lem:uniformsemigroup} that
    \begin{align*}
        \tnorm{\re^{t\hat{L}_{\xi_0}}} \lesssim  \re^{-\frac{1}{2}\min\cur{\theta_2k^2,\theta_3} t}
    \end{align*}
    for all $t \geq 0$ and $\xi_0 \in \bbR^{d-1}$ with $k \leq \abs{\xi_0} \leq k_1$.    
    Setting $\mu = \frac{1}{2}\min\cur{\theta_2 k^2,\theta_3}$, we infer
    \begin{align} \label{eq:midfreq}
        \tnorm{\re^{t\hat{L}_{\xi_0}}} \lesssim  \re^{-\frac{\mu}{2} t - \frac{\mu}{2 k_1^2} \abs{\xi_0}^2 t}
    \end{align}
    for all $t \geq 0$ and $\xi_0 \in \bbR^{d-1}$ with $k \leq \abs{\xi_0} \leq k_1$.   
    
    Combining the high- and mid-frequency estimates~\eqref{eq:highfreq} and~\eqref{eq:midfreq} and setting $\tilde{\mu} = \min\cur{\mu_0 / 2, \mu / (2 k_1^2)}$ (note that $k_1 \geq 1$), then yields
    \begin{align*}
        \tnorm{\re^{t\hat{L}_{\xi_0}}} \lesssim \re^{-\tilde{\mu} t - \tilde{\mu} \abs{\xi_0}^2 t}
    \end{align*}
    for all $t \geq 0$ and $\xi_0 \in \bbR^{d-1}$ with $\abs{\xi_0} \geq k$.   
    In the same way as in Lemma~\ref{lem:semigroup_allfreq}, we finally treat $\smash{\hat{L}_{\xi_0 + \ri \xi_1} - \hat{L}_{\xi_0}}$ as a bounded perturbation, which gives us a constant $\kappa_3 > 0$ such that
    \begin{equation*}
        \tnorm{\re^{t\hat{L}_{\xi_0 + \ri \xi_1}}} \lesssim \re^{-\tilde{\mu} t - \frac{\tilde \mu}{2} \abs{\xi_0}^2 t + \kappa_3 \abs{\xi_1}^2 t}
    \end{equation*}
    for all $t \geq 0$, $\xi_1 \in \bbR^{d-1}$, and $\xi_0 \in \bbR^{d-1}$ with $\abs{\xi_0} \geq k$.
    Thus,~\eqref{eq:semigroup_midhighfreq} holds with $\mu_3 = \tilde{\mu} / 2$.
\end{proof}

\section{Linear theory: semigroup decomposition and estimates} \label{sec:linear}

The linearization $L$ of~\eqref{eq:rde_co} about the planar front $\phi$ is a densely defined, sectorial operator on $\Cub(\bbR^d)$; see, for instance, \cite[Corollary~3.1.9]{Lunardi_analytic_2013}. Consequently, $L$ generates an analytic semigroup $\re^{tL}$.  
Since $L\phi' = 0$, one cannot expect $\re^{tL}$ to exhibit decay in time.  
Nevertheless, Hypotheses~\ref{hyp:spectrum_1d} and~\ref{hyp:spectrum_multid} suggest that the leading-order behavior of $\re^{tL}$ is closely related to the translational mode $\phi'$.
The main result of this section confirms this expectation.

\begin{theorem}[Semigroup decomposition]
    \label{thm:decomposition}
    Let $d_{\bot} > 0$ be as in Proposition~\ref{prop:speccons}. We have
	\begin{equation} \label{eq:spec_decomp_theo}
		\re^{tL}v = \re^{t d_{\bot} \Delta_y} \langle e^*, v\rangle_2 \phi' + S(t)v,
	\end{equation}
    where $\Delta_y$ denotes the Laplacian in the transverse $y$-direction, $e^*$ is as in Section~\ref{subsec:assumptions}, and the remainder $S(t) \colon C_{\mathrm{ub}}(\bbR^d) \to \Cub(\bbR^d)$ satisfies
	\begin{align}    \label{eq:linearest}
   		\norm{S(t)v}_{\infty} &\lesssim (1+t)^{-1} \norm{v}_{\infty}
	\end{align}
    for all $v \in \Cub(\bbR^d)$ and $t \geq 0$.
\end{theorem}

Theorem~\ref{thm:decomposition} shows that the principal part of the semigroup $\re^{tL}$ is $\re^{t d_{\perp}\Delta_y} P_0$, where $P_0$ is given by~\eqref{eq:defP0}. 
In the nonlinear analysis in Section~\ref{sec:nonlinearstab}, we will estimate this term by appealing to classical bounds for the heat semigroup, whose smoothing effect yields decay of derivatives; see, e.g.,~\cite[p.~3537]{matano_large_2011}.

\begin{lemma}[Heat semigroup bounds]
    \label{lem:heat}
    It holds
    \begin{align}
        \label{eq:heat1}
        \norm{\re^{t d_{\bot}\Delta_y} v}_{\infty} + \sqrt{t} \, \norm{\nabla \re^{t d_{\bot}\Delta_y} v}_{\infty} + t\, \norm{\Delta_y \re^{t d_{\bot}\Delta_y} v}_{\infty} &\lesssim \norm{v}_{\infty}
    \end{align}
    for all $v \in \Cub(\bbR^{d-1})$ and $t \geq 0$.
\end{lemma}

In the one-dimensional case, a theorem analogous to Theorem~\ref{thm:decomposition} can be proven by projecting out the translational mode using the spectral projection associated with the isolated eigenvalue $0 \in \sigma(L_0)$.
By Hypothesis~\ref{hyp:spectrum_1d}, this opens up a spectral gap so that the remainder decays exponentially.
In the multidimensional setting however, the eigenvalue $0 \in \sigma(L)$ is embedded in the continuous spectrum, cf.~\eqref{eq:lambdac}, so that a spectral projection is not even well-defined.
This fundamental obstruction makes the analysis much more involved and explains why the decay rate in~\eqref{eq:linearest} is only algebraic (which is nevertheless expected to be optimal).

Outside of the special case considered in~\cite{kapitula_multidimensional_1997}, see Remark~\ref{rem:Didentity} below, the semigroup $\re^{tL}$ is in general not expected to factorize into a longitudinal and a transversal part.
As a result, the proof of Theorem~\ref{thm:decomposition} is substantially more intricate than the linear stability analysis of bistable planar fronts in~\cite{kapitula_multidimensional_1997}. 
A further challenge arises from our objective of establishing pure $L^\infty$-bounds, which, in contrast to the setting of integrable perturbations considered in~\cite{kapitula_multidimensional_1997}, complicates the use of Fourier methods.

\begin{remark}
    \label{rem:Didentity}
    The emergence of $\re^{t d_{\bot}\Delta_y}$ in~\eqref{eq:spec_decomp_theo} can be illuminated by considering the case where $D = I$ as treated in~\cite{kapitula_multidimensional_1997}.
    Here, $L$ decomposes as $L = L_0 + \Delta_y$, where $\Delta_y$ commutes with $L_0$ (where both are considered as operators on $\Cub(\bbR^d)$).
    Consequently, the semigroup factorizes into a longitudinal and a transverse part as
    \begin{equation*}
        \re^{tL} = \re^{t\Delta_y}\re^{tL_0}.
    \end{equation*}
    Applying the one-dimensional spectral projection $P_0$ and using Hypothesis~\ref{hyp:spectrum_1d}, it is straightforward to see that the decomposition
\begin{align} \label{eq:decompKap}
\re^{tL} = \re^{t\Delta_y} P_0 + \mathcal{O}\!\left(\re^{-(\theta_1 / 2) t}\right), \qquad t \geq 0
\end{align}
holds true.
After using~\eqref{eq:defP0} and observing that~\eqref{eq:defdbot} with $D = I$ gives $d_{\bot} = 1$ by~\eqref{eq:normalization}, this yields Theorem~\ref{thm:decomposition}, even with an exponential rate.
Note however that this case is highly nongeneric for several reasons:
\begin{itemize}
    \item $\sigma(L_k) = \sigma(L_0) - k^2$, so Hypothesis~\ref{hyp:spectrum_multid} is implied by Hypothesis~\ref{hyp:spectrum_1d}, precluding the occurrence of transverse instabilities, which are frequently reported in reaction-diffusion models in the literature; see~\cite{carter_criteria_2023,terman_stability_1990,taniguchi_instability_2003} and references therein. 
    \item $\lambda_{\mathrm{c}}(\xi) = \lambda_0(\xi)$ and $P_{\xi} = P_0$ for all $\xi$ in Proposition~\ref{prop:speccons}.
    \item $\re^{t\hat{L}_{\xi}} = \re^{\lambda_0(\xi)t} \re^{tL_0}$, so Lemmas~\ref{lem:semigroup_allfreq}, \ref{lem:semigroup_lowfreq}, and \ref{lem:semigroup_midhighfreq} are trivial.
\end{itemize}
\end{remark}

\subsection{Fourier representation and contour shift} \label{subsec:Fourier_rep}

Let $v \in \Cub(\bbR^{d})$. Taking the Fourier transform in the transverse $y$-variable, we obtain the following representations
\begin{align} \label{eq:Fourier_rep}
\begin{split}
    [\re^{tL}v](t,y,z) &= (2\pi)^{1-d}\int_{\bbR^{d-1}} \int_{\bbR^{d-1}} \re^{\ri \langle \xi, y - \tilde{y}\rangle} [\re^{t\hat{L}_{\xi}}v(\tilde{y},\cdot)](z) \rd \xi \rd \tilde{y}, \\
    [\re^{t d_{\bot}\Delta_y}P_0 v](t,y,z) &= (2\pi)^{1-d}\int_{\bbR^{d-1}} \int_{\bbR^{d-1}} \re^{\ri \langle \xi, y - \tilde{y}\rangle} \re^{-td_{\bot}\lambda_0(\xi)}[P_0 v(\tilde{y},\cdot)](z) \rd \xi \rd \tilde{y},
\end{split}
\end{align}
where $t \geq 0$, $y \in \bbR^{d-1}$, and $z \in \bbR$. Using the semigroup bounds from Lemma~\ref{lem:semigroup_allfreq}, it is seen that the inner integrals in~\eqref{eq:Fourier_rep} converge, whereas convergence of the outer integrals will follow a posteriori from the estimates in the upcoming Propositions~\ref{prop:offdiagonal} and \ref{prop:ondiagonal}.

Our strategy to establishing Theorem~\ref{thm:decomposition} is to bound the inner integral
\begin{align}
    \label{eq:defH}
    [Hv](t,y,\tilde{y},z) \coloneq \int_{\bbR^{d-1}}\re^{\ri \langle \xi, y - \tilde{y}\rangle} \hat{S}_\xi(t) [v(\tilde{y},\cdot)](z) \rd \xi, \qquad \hat{S}_\xi(t) \coloneq \re^{t \hat{L}_{\xi }} - \re^{-td_{\bot} \lambda_0(\xi)}P_0,
\end{align}
arising in the Fourier representation of the remainder $S(t)v = \smash{\re^{tL}v - \re^{t d_{\bot}\Delta_y}P_0v}$, in a way which is integrable in $\tilde{y}$.
A significant challenge is that the factor $\re^{\ri \langle \xi,y - \tilde{y}\rangle}$ does not provide any integrability in $\tilde{y}$ for real-valued $\xi \in \bbR^{d-1}$.
To address this, we analytically continue the frequency into the complex plane by adding a well-chosen imaginary offset proportional to
\begin{equation*}
w \coloneq w(t,y,\tilde{y}) \coloneq \frac{y - \tilde{y}}{ t}.
\end{equation*}
This then leads to the family of operators
\begin{equation}
\label{eq:Heps}
\begin{aligned}
    [H_{\eps}v](t,y,\tilde{y},z) &\coloneq \int_{\bbR^{d-1}}\re^{\ri \langle \xi + \ri \eps w(t,y,\tilde{y}), y - \tilde{y}\rangle} \hat{S}_{\xi + \ri \eps w(t,y,\tilde{y})}[v(\tilde{y},\cdot)] \rd \xi \\
    &= \re^{-\eps \frac{\abs{y - \tilde{y}}^2}{t}}\int_{\bbR^{d-1}}\re^{\ri \langle \xi, y - \tilde{y}\rangle} \hat{S}_{\xi + \ri \eps w(t,y,\tilde{y})}[v(\tilde{y},\cdot)] \rd \xi,
\end{aligned}
\end{equation}
where $\eps \in \bbR$ (we will only use $\eps > 0$), $t > 0$, $y,\yt \in \bbR^{d-1}$, and $z \in \bbR$. By viewing~\eqref{eq:defH} as a superposition of contour integrals, we may use Cauchy's integral theorem and the semigroup estimates from Lemma~\ref{lem:semigroup_allfreq} to shift $\xi \mapsto \xi + \ri \eps w$ and conclude that \begin{align} [H_{\eps}v](t,y,\tilde{y},z) = [Hv](t,y,\tilde{y},z) \label{eq:Cauchy}\end{align} 
for every $\eps > 0$, $t > 0$, $y,\tilde{y} \in \bbR^{d-1}$, and $z \in \bbR$, thereby crucially exploiting analyticity of the Fourier symbol $\smash{\hat{S}_\xi(t)}$ in $\xi$. We will make effective use of~\eqref{eq:Cauchy} in the upcoming propositions to render sufficient localization in $\tilde{y}$.
Additionally, it will be convenient to distinguish between the off-diagonal regime where $\abs{y - \tilde{y}}/t \gg 1$, and the on-diagonal regime where $\abs{y - \tilde{y}} / t \lesssim 1$.

The ideas of complexifying $\xi$ and distinguishing between the off-diagonal and on-diagonal regime to obtain integrable pointwise Green's function bounds can be traced at least as far back as the linear stability analysis of multidimensional viscous shock waves in~\cite{Hoff_pointwise_2002}, although the techniques take quite a different shape in this work. 

\begin{proposition}[Off-diagonal bounds]
    \label{prop:offdiagonal}
    There exist $M,\nu_1 > 0$ such that the estimate
    \begin{equation}
        \label{eq:offdiagonal}
        \abs{[Hv](t,y,\tilde{y},z)} \lesssim \re^{-\nu_1 t - \nu_1 \frac{\abs{y - \tilde{y}}^2}{t}} \norm{v}_{\infty}
    \end{equation}
    holds for all $v \in \Cub(\bbR^{d})$, $y,\tilde{y} \in \bbR^{d-1}$, and $t \geq 1$ with $\abs{y - \tilde{y}} / t \geq M$.
\end{proposition}
\begin{proof}
    Applying Lemma~\ref{lem:semigroup_allfreq}, we obtain the bound
    \begin{align*}
        \tnorm{\re^{t\hat{L}_{\xi + \ri \eps w(t,y,\tilde{y})}}} + \abs{\re^{-td_{\bot}\lambda_0(\xi + \ri \eps w(t,y,\tilde{y}))}} \lesssim \re^{\kappa_1 t - \mu_1 \abs{\xi}^2 t + \kappa_1 \eps^2\abs{w(t,y,\tilde{y})}^2t} 
        = \re^{\kappa_1 t - \mu_1 \abs{\xi}^2 t + \kappa_1 \eps^2\frac{\abs{y - \tilde{y}}^2}{t}}
    \end{align*}
    for all $\eps > 0$, $y,\tilde{y} \in \bbR^{d-1}$, and $t \geq 1$. Thus, taking the absolute value inside the integral in~\eqref{eq:Heps}, setting $\eps = (2\kappa_1)^{-1}$, and integrating over $\xi$, we obtain
    \begin{align*}
        \abs{[H_{\eps}v](t,y,\tilde{y},z)} &\lesssim \re^{\kappa_1 t -\frac{\abs{y - \tilde{y}}^2}{4\kappa_1 t}}  \int_{\bbR^{d-1}}\re^{-\mu_1 \abs{\xi}^2 t} \rd \xi \cdot \|v\|_\infty
        \lesssim \re^{\kappa_1 t -\frac{\abs{y - \tilde{y}}^2}{4\kappa_1 t}} \|v\|_\infty \\
        &\leq \re^{\kappa_1 t -\frac{\abs{y - \tilde{y}}^2}{8 \kappa_1 t} - \frac{M^2}{8 \kappa_1} t} \|v\|_\infty
    \end{align*}
    for all $v \in \Cub(\bbR^{d})$, $y,\tilde{y} \in \bbR^{d-1}$, and $t \geq 1$ with $\abs{y - \tilde{y}} / t \geq M$. Choosing $M = 4 \kappa_1$ and recalling~\eqref{eq:Cauchy}, it follows that~\eqref{eq:offdiagonal} holds with $\nu_1 = \min\cur{\kappa_1,(8\kappa_1)^{-1}}$.
\end{proof}

\begin{proposition}[On-diagonal bounds]
    \label{prop:ondiagonal}
    Fix $M > 0$.
    Then, there exists $\nu_2 > 0$ such that the estimate
    \begin{equation}
        \label{eq:ondiagonal}
        \abs{[Hv](t,y,\tilde{y},z)} \lesssim t^{-(d+1)/2} \re^{-\nu_2 \frac{\abs{y - \tilde{y}}^2}{t}} \norm{v}_\infty
    \end{equation}
    holds for all $v \in \Cub(\bbR^{d})$, $y,\tilde{y} \in \bbR^{d-1}$, and $t \geq 1$ with $\abs{y - \tilde{y}} / t \leq M$.
\end{proposition}
\begin{proof}
    As in the proof of Proposition~\ref{prop:offdiagonal}, it suffices to estimate $[H_{\eps}v](t,y,\tilde{y},z)$ for some appropriately chosen value of $\eps$ and subsequently use~\eqref{eq:Cauchy}.
    
    Before we start, we let $k$ be as in Lemma~\ref{lem:semigroup_lowfreq} and let $B_0$ be the ball in $\bbR^{d-1}$ of radius $k$ centered at the origin. Using~\eqref{eq:Heps}, we estimate
    \begin{align}
        \label{eq:ondiagonal_T1234}
        \re^{\eps \frac{\abs{y - \tilde{y}}^2}{t}}\abs{[H_{\eps}v](t,y,\tilde{y},z)} \lesssim \|v\|_\infty \int_{\bbR^{d-1}}\tnorm{\hat{S}_{\xi + \ri \eps w(t,y,\yt)} } \rd \xi \lesssim \|v\|_\infty(T_1 + T_2 + T_3),
    \end{align}
    for all $\eps > 0$, $v \in \Cub(\bbR^{d})$, $y,\tilde{y} \in \bbR^{d-1}$, $z \in \bbR$, and $t \geq 1$, where we have introduced the quantities
    \begin{align*}
        T_1 &\coloneq \int_{\bbR^{d-1} \setminus B_0}\tnorm{\re^{t\hat{L}_{\xi + \ri \eps w(t,y,\yt)}}} + \abs{\re^{-td_{\bot}\lambda_0(\xi + \ri \eps w(t,y,\yt))}} \rd \xi, \\
        T_2 &\coloneq \int_{B_0}\tnorm{\re^{t\hat{L}_{\xi + \ri \eps w(t,y,\yt)}}(I - P_{\xi + \ri \eps w(t,y,\yt)})} \rd \xi, \\
        T_3 &\coloneq \int_{B_0}\tnorm{\re^{t\hat{L}_{\xi + \ri \eps w(t,y,\yt)}}P_{\xi + \ri \eps w(t,y,\yt)} - \re^{-td_{\bot}\lambda_0(\xi + \ri\eps w(t,y,\yt))}P_0} \rd \xi,
    \end{align*}
    which we will estimate one by one.
        
    It holds by Lemma~\ref{lem:semigroup_midhighfreq} that
    \begin{align*}
        \tnorm{\re^{t\hat{L}_{\xi + \ri \eps w}}} + \abs{\re^{-td_{\bot}\lambda_0(\xi+ \ri \eps w)}} 
        \lesssim \re^{-\mu_3 t - \mu_3\abs{\xi}^2 t + \kappa_3 \eps^2 \abs{w}^2 t} 
        \lesssim \re^{-\mu_3 t - \mu_3\abs{\xi}^2 t + \kappa_3 \eps^2 M^2 t}
    \end{align*}
    for all $t \geq 0$, $\eps > 0$, $\xi \in \bbR^{d-1} \setminus B_0$, and $w \in \bbR^{d-1}$ with $\abs{w} \leq M$. Thus, it follows, after  integrating over $\xi$, that
    \begin{align*}
        T_1 \lesssim \re^{-\frac{\mu_3}{2} t}
    \end{align*}
    for all $\eps > 0$, $y,\tilde{y} \in \bbR^{d-1}$, and $t \geq 1$ with $\abs{y - \tilde{y}} / t \leq M$ and $\eps \leq \sqrt{\mu_3 / (2 \kappa_3 M^2)}$.
    
    We proceed with bounding $T_2$ and $T_3$. By Lemma~\ref{lem:semigroup_lowfreq} we have
    \begin{align*}
        \tnorm{\re^{t\hat{L}_{\xi + \ri \eps w}}(I - P_{\xi + \ri \eps w})} &\lesssim \re^{-\mu_2 t},\\
        \tnorm{\re^{t\hat{L}_{\xi + \ri \eps w}}P_{\xi + \ri \eps w} - \re^{-td_{\bot}\lambda_0(\xi + \ri \eps w)}P_0}
            &\lesssim t^{-1} \re^{-\mu_2\abs{\xi}^2 t + \kappa_2 \eps^2 \abs{w}^2 t}
    \end{align*}
    for all $t \geq 1$, $\eps \in (0,kM^{-1})$, $\xi\in B_0$, and $w \in \bbR^{d-1}$ with $\abs{w} \leq M$. Integrating these bounds over $\xi \in B_0$ gives 
    \begin{align*}
        T_2 &\lesssim \re^{-\mu_2 t},\\
        T_3 &\lesssim t^{-1}t^{-(d-1)/2} \re^{\kappa_2 \eps^2 \abs{w(t,y,\yt)}^2 t} \leq t^{-(d+1)/2} \re^{\frac{\eps \abs{y - \tilde{y}}^2}{2 t}},
    \end{align*}
    for $\eps \in (0,\min\cur{k M^{-1}, (2\kappa_2)^{-1}})$, $y,\tilde{y} \in \bbR^{d-1}$, and $t \geq 1$ with $\abs{y - \tilde{y}} / t \leq M$. Thus, choosing $\eps > 0$ such that the established bounds on $T_1,T_2$, and $T_3$ hold simultaneously and looking back at~\eqref{eq:Cauchy} and~\eqref{eq:ondiagonal_T1234}, we obtain
    \begin{align*}
        \abs{[Hv](t,y,\tilde{y},z) } = \abs{[H_{\eps}v](t,y,\tilde{y},z) } \lesssim t^{-(d+1)/2}\re^{-\frac{\eps \abs{y - \tilde{y}}^2}{2 t}}\|v\|_\infty
    \end{align*}
    for all $v \in \Cub(\bbR^{d})$, $y,\tilde{y} \in \bbR^{d-1}$, and $t \geq 1$ with $\abs{y - \tilde{y}} / t \leq M$. We conclude that~\eqref{eq:ondiagonal} holds with $\nu_2 = \eps / 2$.
\end{proof}

We now combine the off-diagonal and on-diagonal bounds to prove Theorem~\ref{thm:decomposition}.

\begin{proof}[Proof of Theorem~\ref{thm:decomposition}]
    Let $M,\nu_1 > 0$ be as in Proposition~\ref{prop:offdiagonal}, let $\nu_2 > 0$ be as in Proposition~\ref{prop:ondiagonal}, and set $\nu_3 = \min\cur{\nu_1,\nu_2}$.
    For $t \geq 1$ and $y \in \bbR^{d-1}$, define $V_{t,y}$ to be the set of all $\tilde{y} \in \bbR^{d-1}$ which satisfy $\abs{y - \tilde{y}} / t \leq M$.
    Then, we have by~\eqref{eq:Fourier_rep},~\eqref{eq:defH},~\eqref{eq:offdiagonal}, and~\eqref{eq:ondiagonal}:
    \begin{align*}
        \abs{[\re^{tL}v - \re^{td_{\bot}\Delta_y}P_0 v](t,y,z)}
        &\leq \int_{\bbR^{d-1}} \abs{[Hv](t,y,\tilde{y},z)} \rd \tilde{y} \\
        &= \int_{\bbR^{d-1} \setminus V_{t,y}} \abs{[Hv](t,y,\tilde{y},z)} \rd \tilde{y} 
        + \int_{V_{t,y}} \abs{[Hv](t,y,\tilde{y},z)} \rd \tilde{y} \\
        &\lesssim \frac{\|v\|_\infty}{\re^{\nu_3 t}}\int_{\bbR^{d-1}\setminus V_{t,y}} \re^{-\nu_3\frac{\abs{y - \tilde{y}}^2}{t}} \rd \tilde{y}
        + \frac{\|v\|_\infty}{t^{(d+1)/2}} \int_{V_{t,y}} \re^{-\nu_3\frac{\abs{y - \tilde{y}}^2}{t}} \rd \tilde{y} \\
        &\lesssim  \frac{\|v\|_\infty}{t^{(d+1)/2}} \int_{\bbR^{d-1}}\re^{-\nu_3 \frac{\abs{y - \tilde{y}}^2}{t}} \rd \tilde{y} 
        \lesssim \frac{\|v\|_\infty}{t}
    \end{align*}
    for all $v \in \Cub(\bbR^d)$ $t \geq 1$, $y \in \bbR^{d-1}$, and $z \in \bbR$,    
    which establishes~\eqref{eq:linearest} for $t \geq 1$.
    For short times $t \in [0,1]$, the bound~\eqref{eq:linearest} follows from standard semigroup theory; see~\cite[Proposition~I.5.5]{engel_oneparameter_2000}.
\end{proof}

\section{Nonlinear theory:  tracking scheme}
\label{sec:modulation}
In this section, we set up the tracking scheme, which we will employ to close a nonlinear stability argument and prove our main results. In particular, we introduce a spatiotemporal modulation of the front to capture its neutral translational dynamics.

\subsection{Front modulation}
\label{subsec:shift}

Our objective is to describe the dynamics of the solution $u(t)$ to~\eqref{eq:rdefull}, where $E_0 \coloneq \|u_0 - \phi\|_\infty$ is sufficiently small.
Inspired by the nonlinear stability theory for periodic waves~\cite{Doelman_dynamics_2009,johnson_nonlinear_2011}, we introduce the \emph{inverse modulation Ansatz}
\begin{equation}
 	\label{eq:modulation}
	\tilde{u}(t,y,z) = u(t,y,z + ct + \sigma(t,y)),
\end{equation}
where $\sigma \colon [0,\infty) \times \bbR^{d-1} \to \bbR$ is a smooth modulation function that we will choose a posteriori so as to track the leading-order position of the front interface. Concretely, $\sigma$ is selected so that the \emph{inverse-modulated perturbation}
\begin{equation}
    \label{eq:vurel}
    v(t,y,z) = u(t,y,z + ct + \sigma(t,y)) - \phi(z)
\end{equation}
exhibits temporal decay. Since $\|v(t)\|_\infty$ is equal to the $L^\infty$-norm of the \emph{forward-modulated perturbation}
\begin{align*} 
\mathring{v}(t,y,z) = u(t,y,z) - \phi(z - ct - \sigma(t,y)),
\end{align*}
it suffices to control $v(t)$ to establish the key estimate~\eqref{eq:thmest4}. This decay estimate holds precisely if $\sigma$ captures the neutral translational dynamics of the front to leading order.

Using~\eqref{eq:modulation}, we obtain the relations
\begin{align*}
	u(t,y,z) &= \tilde{u}(t,y,z-ct-\sigma), \\
	u_t(t,y,z) &= \tilde{u}_t(t,y,z-ct-\sigma) - (c+\sigma_t)\tilde{u}_z(t,y,z-ct-\sigma), \\
	u_{y_i}(t,y,z) &= \tilde{u}_{y_i}(t,y,z-ct-\sigma) - \sigma_{y_i} \tilde{u}_z(t,y,z-ct-\sigma), \\
	u_{y_iy_i}(t,y,z) &= \tilde{u}_{y_iy_i}(t,y,z-ct-\sigma) - 2\sigma_{y_i} \tilde{u}_{y_iz}(t,y,z-ct-\sigma) \\
	&\qquad - \, \sigma_{y_iy_i}\tilde{u}_z(t,y,z-ct-\sigma) + \sigma_{y_i}^2 \tilde{u}_{zz}(t,y,z-ct-\sigma)
\end{align*}
for $i = 1,\ldots,d-1$, where we have abbreviated $\sigma(t,y)$ as $\sigma$. Substituting these identities into~\eqref{eq:rde} and shifting coordinates back gives
\begin{align}
	\label{eq:utildepde}
	\partial_t\tilde{u} - (c + \sigma_t)\tilde{u}_z
	= D\Delta \tilde{u} + f(\tilde{u})
    - \Delta_y \sigma\, D \tilde{u}_z
    - 2 \nabla \sigma \cdot \nabla D\tilde{u}_z
    + |\nabla \sigma|^2 D \tilde{u}_{zz}
\end{align}
with all terms evaluated at $(t,y,z)$.

Although the viscous Hamilton--Jacobi equation~\eqref{eq:HJ} provides an effective description of $\sigma$, it is neither accurate enough to yield the decay rate required in~\eqref{eq:thmest4} nor sufficiently regular near $t=0$ to serve as a convenient starting point. Instead, we augment~\eqref{eq:HJ} with a mild forcing term $g$, and then prescribe $\sigma$ via~\eqref{eq:pdesigma}, thus imposing the initial condition $\sigma(0) = 0$. Substituting~\eqref{eq:pdesigma} into~\eqref{eq:utildepde} yields
\begin{equation}
\label{eq:utildepde2}
\begin{aligned}
\partial_t\tilde{u} &= D\Delta \tilde{u} + c\,\tilde{u}_z + f(\tilde{u}) \\
&\qquad + \Delta_y \sigma\,(d_{\bot}\tilde{u}_z - D\tilde{u}_z)
    + |\nabla \sigma|^2\,(D\tilde{u}_{zz} + \tfrac{1}{2}c\,\tilde{u}_z)
    - 2\nabla \sigma \cdot \nabla D\tilde{u}_z
    - g\,\tilde{u}_z.
\end{aligned}
\end{equation}
Subsequently inserting~\eqref{eq:vurel} into~\eqref{eq:utildepde2} and using~\eqref{eq:modulation} together with $\sigma(0)=0$ to determine the initial condition, we obtain that the inverse-modulated perturbation $v$ satisfies
\begin{equation}
\begin{aligned}
    \label{eq:pdev}
	\partial_t v &= L v
    + \Delta \sigma \, (d_{\bot}\phi' - D\phi')
    + |\nabla \sigma|^2 \, (D\phi'' + \tfrac12 c \, \phi')
    + N(v,\sigma,g) - g\, \phi', \\
    v(0) &= u_0 - \phi,
\end{aligned}
\end{equation}
where the nonlinearity is given by
\begin{equation}
    \label{eq:defN}
\begin{aligned}
	N(v,\sigma,g) &= f(\phi+v) - f(\phi) - f'(\phi)v
        + \Delta \sigma\, (d_{\bot}v_z - Dv_z)
        + |\nabla \sigma|^2\,(D v_{zz} + \tfrac12 c \,v_z) \\
	&\qquad - 2\nabla \sigma \cdot \nabla D v_z
        - g \, v_z.
\end{aligned}
\end{equation}
Note that there is no term of the form $\nabla \sigma \cdot \nabla \phi'$, since $\sigma$ is independent of $z$ and $\phi$ is independent of $y$. Using Taylor's theorem, we readily establish the following.

\begin{lemma}[Bound on nonlinearity]
	\label{lem:Nestimate1}
    Fix a constant $C > 0$. It holds
	\begin{align}
        \label{eq:Nest1}
		\norm{N(v,\sigma,g)}_{\infty} &\lesssim \norm{v}_{\infty}^2 + \norm{v}_{C^2} \, \bra[\big]{\norm{\Delta \sigma}_{\infty} + \norm{\nabla \sigma}_{\infty} + \norm{\nabla \sigma}_{\infty}^2 + \norm{g}_{\infty}}
	\end{align}
    for all $g \in C_{\mathrm{ub}}(\bbR^{d-1})$, $\sigma \in C_{\mathrm{ub}}^2(\bbR^{d-1})$, and $v \in C_{\mathrm{ub}}^2(\bbR^d)$ with $\|v\|_\infty \leq C$.
\end{lemma}

\begin{remark} \label{rem:invvsforward}
Spatiotemporal front modulation is also used in the nonlinear stability analysis of planar fronts under localized perturbations in~\cite{kapitula_multidimensional_1997}. There, one directly estimates the forward-modulated perturbation $\mathring{v}(t)$. In contrast to~\eqref{eq:pdev}, the resulting equation for $\mathring{v}(t)$ contains nonlinear terms of the form $\left(f'(\phi(\cdot+\sigma)) - f'(\phi)\right)\mathring{v}$. In our setting of $\Cub$-perturbations, available decay rates are too weak to control such contributions, which decay only like $t^{-1/2}$. Thus, in our analysis it is advantageous to work with the inverse-modulated perturbation instead.
\end{remark}

\subsection{Forcing}
\label{subsec:forcing}
We specify the forcing term in~\eqref{eq:pdesigma}, which determines $\sigma$. To this end, we use the temporal cut-off function $\varrho$, introduced in Section~\ref{subsec:notation}, which is smooth, nonnegative, has unit integral, and is supported on $[1/4,3/4]$. For $t \geq 0$, we then set
\begin{align} \label{eq:defgnew} 
\begin{split}
g(t) &= \varrho(t) \re^{t d_{\bot}\Delta_y} \langle e^*, u_0 - \phi \rangle_2\\ 
&\qquad + \, \sum_{k \in \bbN} \varrho(t - k)\re^{(t-k)d_{\bot}\Delta_y} \int_{k-1}^k \re^{(k-r)d_{\bot}\Delta_y} \langle e^*, N(v(r),\sigma(r),g(r))\rangle_2 \rd r . 
\end{split}
\end{align}
We will see in Section~\ref{sec:nonlinearstab} that this specific choice of $g$ cancels the critical nonlinear contributions in equation~\eqref{eq:pdev} in a gentle way. In particular, it enables us to establish decay of $v$, ensuring that the modulation $\sigma$ indeed captures the leading-order dynamics of the front interface. 

The properties of $\varrho$ guarantee that $g$ is smooth in both space and time. We will show that this smoothness is inherited by the modulation $\sigma$. Another useful feature of~\eqref{eq:defgnew} is that, for $t\in[0,1]$ and $n\in \bbN$, the quantity $g(n+t)$ is fully determined by $N(v(s),\sigma(s),g(s))$ on $[0,n]$.  As a result, $g$ and $\sigma$ can be extended in the nonlinear iteration without introducing an additional fixed-point argument, providing a simpler local analysis than that required by previous stability approaches based on spatiotemporal modulation; cf.~\cite[Appendix~B]{derijk_nonlinear_2024}.

Since $\varrho$ is supported on $[1/4,3/4]$, the nonlinearity in~\eqref{eq:defgnew} is always preceded by the smoothing heat semigroup $\smash{\re^{\frac18 d_\perp \Delta_y}}$.
The inner product against $e^*$ provides additional smoothing in the $z$-direction.
These smoothing effects are captured in the following result.

\begin{lemma}[Bound on smoothed nonlinearity]
	\label{lem:Nestimate2}
    Fix a constant $C > 0$. We have
    \begin{align}
        \label{eq:Nest2}
        \norm{\re^{\frac{1}{8} d_{\bot}\Delta_y} \langle e^*,N(v,\sigma,g)\rangle_2}_{\infty} &\lesssim \norm{v}_{\infty}\bra[\big]{\norm{v}_{\infty} + \norm{\Delta \sigma}_{\infty} + \norm{\nabla \sigma}_{\infty} + \norm{\nabla \sigma}_{\infty}^2 + \norm{g}_{\infty}}
	\end{align}
    for all $g \in C_{\mathrm{ub}}(\bbR^{d-1})$, $\sigma \in C_{\mathrm{ub}}^2(\bbR^{d-1})$, and $v \in C_{\mathrm{ub}}^2(\bbR^d)$ with $\|v\|_\infty \leq C$.
\end{lemma}
\begin{proof}
    We first rewrite
    \begin{align*}
        \langle e^*,N(v,\sigma,g)\rangle_2 &= \langle e^*, f(\phi + v) - f(\phi) - f'(\phi)v \rangle_2 - \Delta \sigma \, \langle e^*_z, d_{\bot}v - D v \rangle_2  \\
        &\qquad + \, \abs{\nabla \sigma}^2 \, \bra[\big]{\langle e^*_{zz}, D v\rangle_2 
         - \tfrac{1}{2}c \, \langle e^*_z,  v\rangle_2}  + 2\, \nabla \cdot \bra[\big]{\nabla \sigma \, \langle e^*_z, Dv\rangle_2}\\ 
        &\qquad - \, 2\, \Delta \sigma \, \langle e^*_z,D v\rangle_2 + g \, \langle e^*_z, v\rangle_2,
    \end{align*}
    using integration by parts with respect to $z$ and the product rule.
    Using integrability of $e^*,e^*_z$, and $e^*_{zz}$, every term can now be estimated in the obvious way except for the last term on the second line.
    Here, we use the smoothing effect of the heat kernel to see that
    \begin{equation*}
        \norm{\re^{\frac{1}{8} d_{\bot}\Delta_y} \sqr[\big]{ \nabla \cdot (\nabla \sigma\langle e^*_z, D v\rangle_2 )}}_{\infty} \lesssim \norm{ \nabla \sigma \, \langle e^*_z,D v\rangle_2}_{\infty}
    \end{equation*}
    for all $v \in C_{\mathrm{ub}}^1(\bbR^d)$ and $\sigma \in C_{\mathrm{ub}}^2(\bbR^{d-1})$; see Lemma~\ref{lem:heat}. The estimate~\eqref{eq:Nest2} follows.
\end{proof}

\section{Nonlinear theory: stability analysis}
\label{sec:nonlinearstab}

This section is devoted to two a priori estimates that form the core of our nonlinear stability argument. 
Assuming the existence of a solution $(u,\sigma,v,g)$ to the system given by~\eqref{eq:rdefull},~\eqref{eq:pdesigma},~\eqref{eq:pdev}, and~\eqref{eq:defgnew} on a time interval $[0,T]$, we combine the linear estimates from Theorem~\ref{thm:decomposition} and Lemma~\ref{lem:heat} with the nonlinear bounds in Lemmas~\ref{lem:Nestimate1} and~\ref{lem:Nestimate2} to derive decay estimates for $v$, $\sigma$, $g$, and the nonlinearity $N(v,\sigma,g)$. 

\subsection{Decay rates and optimality} \label{subsec:expect_decay}

Before turning to the analysis, let us first discuss the decay rates that could be expected. 
Linearizing~\eqref{eq:pdesigma} yields the heat equation $\partial_t \sigma = d_\bot \Delta \sigma$, which, by Lemma~\ref{lem:heat}, suggests that \begin{align} \label{eq:sigmadecay}
	\sigma \sim E_0, \qquad \nabla \sigma \sim E_0 t^{-1/2}, \qquad \Delta \sigma \sim E_0 t^{-1},
\end{align}
where we recall $E_0 = \norm{u_0 - \phi}_\infty$. 
These decay rates are, at least at the linear level, optimal and imply that the nonlinearity in~\eqref{eq:pdev} decays at rate $t^{-1}$. 
Consequently, one cannot expect $v$ to decay faster than $E_0 t^{-1}$. 
In practice, integrating a $t^{-1}$ source term introduces an additional logarithm, and we will ultimately establish that $v$ decays at rate $t^{-1}\log(t)$. 
Combining this with~\eqref{eq:sigmadecay} and Lemma~\ref{lem:Nestimate1} indicates that $N(v,\sigma,g)$ should decay at the rate $t^{-3/2}\log(t)$. 
However, because our derivative bounds on $v$ rely on tame estimates obtained via interpolation, we establish a slightly weaker (but still sufficient) decay rate of $t^{-5/4}$ for $N(v,\sigma,g)$. Since derivatives of $v$ can be absorbed by the smoothing action of the heat semigroup, cf.~Lemma~\ref{lem:Nestimate2}, we may use the sharper rate $t^{-3/2}\log(t)$ for the nonlinearity appearing in the equation~\eqref{eq:defgnew} for $g$, which, together with the compact support of $\varrho$, suggests that $g$ decays at the rate $t^{-3/2}\log(t)$.

\subsection{Cole--Hopf transform}
\label{subsec:CH}
To deal with the critical nonlinearity $\smash{\tfrac12 c\,\abs{\nabla \sigma}^2}$ in~\eqref{eq:pdesigma} we will make use of the Cole--Hopf transform, which we now introduce.
To lighten the coming notation, we begin by defining the parameter
\begin{align}
    \label{eq:defbeta}
    \beta \coloneq \frac{c}{2d_\perp}.
\end{align}
Setting $\mathcal{I}_\beta \coloneq \{x \in \bbR : \beta x > -1\}$, the Cole--Hopf transform $\Psi_\beta \colon \bbR \to \mathcal{I}_\beta$ and its inverse $\smash{\Psi_\beta^{-1}} \colon \mathcal{I}_\beta \to \bbR$ are given by 
\begin{subequations}
\begin{align}
 \label{eq:CH}
 \Psi_\beta(x) &= \begin{cases} \frac{1}{\beta} (\re^{\beta  x}-1), & \beta \neq 0, \\  x, & \beta = 0,\end{cases}\\
\label{eq:CHinv}
 \Psi_\beta^{-1}(x) &= \begin{cases} \frac{1}{\beta} \log(1+\beta x), & \beta \neq 0, \\  x, & \beta = 0.\end{cases}
\end{align}
\end{subequations}
A crucial property of the Cole--Hopf transform is that upon introducing the variable $\xi = \Psi_\beta(\sigma)$, equation~\eqref{eq:pdesigma} transforms into
\begin{equation}
\label{eq:pdexibar}
\begin{aligned}
\partial_t \xi &= d_{\bot}\Delta \xi - g(1 + \beta \xi), \\
\xi(0) &= 0,
\end{aligned}
\end{equation}
which no longer contains the critical nonlinearity.
In order to solve~\eqref{eq:pdesigma} it thus suffices to solve~\eqref{eq:pdexibar} and take $\sigma = \smash{\Psi_\beta^{-1}}(\xi)$ afterwards (taking into account that $\smash{\Psi_\beta^{-1}}(x)$ is only well-defined for $x$ in the interval $\mathcal{I}_\beta$).
We will generally refer to $\xi$ as the Cole--Hopf variable.
\begin{remark}
    When $c = 0$, the coefficient of the critical nonlinearity in~\eqref{eq:pdesigma} vanishes and no transformation is needed. This is reflected in the fact that, for $\beta = 0$, both $\Psi_\beta$ and its inverse reduce to the identity and equations~\eqref{eq:pdesigma} and~\eqref{eq:pdexibar} coincide. Thus, the above definition of the Cole--Hopf transform allow us to treat the cases $c = 0$ and $c \neq 0$ simultaneously within a unified framework.
\end{remark}

\subsection{Analysis of forced viscous Hamilton--Jacobi equation}

Our first a priori estimate concerns the modulation function $\sigma$, which solves the forced viscous Hamilton--Jacobi equation~\eqref{eq:pdesigma}.
Applying the Cole--Hopf transform by setting $\xi = \Psi_\beta(\sigma)$ removes the critical nonlinearity from~\eqref{eq:pdesigma} and allows us to control the solution in terms of the forcing $g$.  
For this purpose, we assume that $g$ decays at the integrable rate $t^{-5/4}$, which is weaker than the expected decay rate $t^{-3/2}\log(t)$, cf.\ Section~\ref{subsec:expect_decay}, but suffices for our purposes.

\begin{lemma}[A priori estimate on modulation]
    \label{lem:xiAP}
Let $C > 0$ and $k \in \bbN_0$. Then, there exist  constants $C_0, M_0 > 0$ such that for each $T > 0$, $M \in (0,M_0]$, and smooth $g \colon [0,T] \times \bbR^{d-1} \to \bbR$ satisfying
\begin{equation}
\label{eq:gestassume}
\norm{g(t)}_{C^k} \leq C M (1+t)^{-5/4}, \qquad t \in [0,T],
  \end{equation}
there exists a unique smooth solution $\xi \colon [0,T] \times \bbR^{d-1} \to \bbR$ to~\eqref{eq:pdexibar}
obeying the a priori estimate
\begin{equation}
\label{eq:xiap}
\norm{\xi(t)}_{C^k} \leq C_0 M
\end{equation}
for all $t \in [0,T]$. 
Moreover, $\sigma \coloneq \smash{\Psi_\beta^{-1}}(\xi)$
is the unique smooth solution to~\eqref{eq:pdesigma} on $[0,T]$.
\end{lemma}
\begin{proof}
Existence and uniqueness of a smooth solution $\xi$ to~\eqref{eq:pdexibar} follow from standard parabolic theory; see~\cite[Chapter~4]{Lunardi_analytic_2013}. Such a function $\xi$ satisfies the Duhamel formula
\begin{align}
    \label{eq:xiduhamellem}
    \xi(t) = -\int_0^t \re^{(t-s)d_{\bot}\Delta_y} g(s)(1+\beta\xi(s))\rd s
\end{align}
for $t \in [0,T]$. Using~\eqref{eq:gestassume}, integrability of $(1+t)^{-5/4}$, and Lemma~\ref{lem:heat}, we obtain the estimate
\begin{equation*}
    \norm{\xi(t)}_{C^k} \leq K_0 M + K_0 \int_0^t (1+s)^{-5/4} \norm{\xi(s)}_{C^k} \rd s
\end{equation*}
for $t \in [0,T]$, where $K_0 > 0$ is a constant, depending only on $C$ and $k$. An application of Gr\"onwall's inequality then yields a constant $C_0 > 0$, depending only on $C$ and $k$, such that~\eqref{eq:xiap} holds. 

For $c = 0$ the final claim is immediate, since in this case $\sigma \equiv \xi$ and equation~\eqref{eq:pdesigma} coincides with~\eqref{eq:pdexibar}. 
For $c \neq 0$, set $M_0 = 1/(2C_0\beta)$. Then,~\eqref{eq:xiap} implies $-1/2 \leq \beta\xi \leq 1/2$ whenever $M \leq M_0$. 
Consequently, $\sigma \coloneq \smash{\Psi_\beta^{-1}}(\xi) = \beta^{-1}\log(1+\beta\xi)$ is well-defined, and a direct computation shows that $\sigma$ solves~\eqref{eq:pdesigma}. Since solutions $\sigma$ to~\eqref{eq:pdesigma} and solutions $\xi$ to~\eqref{eq:pdexibar} are related through the Cole--Hopf transform, uniqueness of $\sigma$ follows from uniqueness of $\xi$.
\end{proof}

\subsection{Nonlinear iteration argument}

We proceed with establishing a priori bounds on $v$, $\sigma$, $g$, and the nonlinearity $N(v,\sigma,g)$, which reflect the expected decay rates as formally derived in Section~\ref{subsec:expect_decay}. To this end, we use the following template function
\begin{equation}
    \label{eq:template}
    \eta(t) = \sup_{s \in [0,t]} \bra[\Big]{E_0 + \norm{\sigma(s)}_{\infty} + E_0^{-1}(1+s)^{5/4} \norm{N(v(s),\sigma(s),g(s))}_{\infty} },
\end{equation}
recalling that $E_0 \coloneq \norm{u_0 - \phi}_{\infty}$.
As an input to the nonlinear iteration, we will assume that $\eta(t)$ is bounded.
With this assumption, we derive bounds on $v$, $\sigma$, and $g$, together with a bound on $\eta(t)$ which is stronger than the assumed one, allowing us to close the nonlinear iteration.
The output of the iteration is stated in the following lemma.

\begin{lemma}[Nonlinear iteration]
    \label{lem:key}
    Let $k \in \bbN_0$. 
    There exist constants $C,\delta > 0$ such that for any $T > 0$, $u_0 \in \Cub(\bbR^d)$, smooth $\sigma,g \colon [0,T] \times \bbR^{d-1} \to \bbR$, and
    \begin{align*}
    v \in C\big([0,T],\Cub(\bbR^d)\big) \cap C^\infty\big((0,T] \times \bbR^d\big)
    \end{align*}
    which satisfy~\eqref{eq:pdesigma},~\eqref{eq:pdev}, ~\eqref{eq:defgnew}, and $\eta(t) \leq 1$ for $t \in [0,T]$, the a priori estimates
   	\begin{subequations}
    \begin{align}
    \label{eq:apkeyest}
    \eta(t) &\leq C E_0^{3/4}, \\ 
    \label{eq:vkeyest}
    \norm{v(t)}_{\infty} &\leq C E_0 (1+t)^{-1}\log(2+t),\\
    \label{eq:gbetterest}
    \norm{g(t)}_{C^k} &\leq C\left( \varrho(t) E_0  + E_0^2 (1+t)^{-3/2} \log(2+t)\right),
    \end{align}
    and
    \begin{align}
        \label{eq:sigmaest}
	\norm{\sigma(t)}_\infty + (1+t)^{1/2} \norm{\nabla \sigma(t)}_\infty + (1+t) \norm{\Delta \sigma(t)}_\infty &\leq C E_0
    \end{align}
    \end{subequations}
    hold for all $t \in [0,T]$.
    Moreover, if $E_0 \leq \delta$, we have $\eta(t) \leq 1$ for all $t \in [0,T]$.
\end{lemma}

To bound nonlinear contributions in the Duhamel formulation of $v$, $\sigma$, and $g$, we use the following classical estimate on convolution integrals.

\begin{lemma}[Convolution estimate]
    \label{lem:convolution}
    Let $\alpha,\gamma \geq 0$. We have
    \begin{align*}
    \int_0^t (1+t-s)^{-\alpha}(1+s)^{-\gamma} \rd s \lesssim \begin{cases}
            (1+t)^{-1}\log(2+t), & \alpha = \gamma = 1, \\
            (1+t)^{\max\cur{-\alpha,-\gamma,-\alpha-\gamma + 1}}, & \text{\normalfont otherwise}
        \end{cases}
    \end{align*}
    for all $t \geq 0$.
\end{lemma}

\begin{proof}[Proof of Lemma~\ref{lem:key}]
Throughout the proof, we denote by $C > 0$ any constant, which is independent of $t$, $T$, $u_0$, $E_0$, $u$, $v$, $\sigma$, and $g$. 

Let $t \in [0,T]$ be such that $\eta(t) \leq 1$. By definition of $\eta(t)$, we have
\begin{align}
    \label{eq:Nbaseest}
    \norm{N(s)}_{\infty} \leq E_0 (1+s)^{-5/4}
\end{align}
for $s \in [0,t]$, where we have abbreviated
\begin{align*}
N(s) \coloneq N(v(s),\sigma(s),g(s)).
\end{align*}
Applying Lemma~\ref{lem:heat} and recalling that $\varrho$ is supported on $[1/4,3/4]$, we establish
\begin{align*}
\begin{split}
     \norm{g(\tau)}_{C^k} &\leq C \varrho(\tau) E_0,\qquad  \|g(n+\tau)\|_{C^k} \leq C\sup_{r\in[n-1,n]}\norm{N(r)}_{\infty}
\end{split}
\end{align*}
for all $\tau \in[0,1]$ and $n\in\bbN$ such that $\tau, n + \tau \in [0,t]$. In particular, combining the latter with~\eqref{eq:Nbaseest} implies
\begin{align}
    \label{eq:gest}
    \norm{g(s)}_{C^9} \leq C E_0 (1+s)^{-5/4}
\end{align}
for $s \in [0,t]$, where the choice for $k = 9$ is motivated by the interpolation argument in the upcoming Step~5 of the proof.

\emph{Step 1: Estimating $\sigma$.}
Recalling Section~\ref{subsec:CH}, we introduce the Cole--Hopf variable $\xi \coloneq \Psi_\beta(\sigma)$ so that $\xi$ solves~\eqref{eq:pdexibar} and obeys the associated Duhamel formulation
\begin{align}
    \label{eq:xiduhamel}
    \xi(s) = -\int_0^s \re^{(s-r)d_{\bot}\Delta_y} g(r)(1+\beta\xi(r))\rd r
\end{align}
for $s \in [0,t]$. Using~\eqref{eq:gest} and Lemma~\ref{lem:xiAP}, we then obtain 
\begin{equation}
    \label{eq:xiestC9}
    \norm{\xi(s)}_{C^9} \leq C E_0
\end{equation}
for $s \in [0,t]$. Combining this with~\eqref{eq:gest} and Lemma~\ref{lem:heat} also gives
\begin{align*}
    \norm{\nabla \re^{(s-r)d_{\bot}\Delta_y}g(r)(1+\beta\xi(r))}_{\infty} &\leq C E_0 (1+s-r)^{-1/2} (1+r)^{-5/4}, \\
    \norm{\Delta \re^{(s-r)d_{\bot}\Delta_y}g(r)(1+\beta\xi(r))}_{\infty} &\leq C E_0 (1+s-r)^{-1} (1+r)^{-5/4}
\end{align*}
for $r,s \in [0,t]$ with $r \leq s$. Substituting this into~\eqref{eq:xiduhamel} and using Lemma~\ref{lem:convolution} gives the following estimate
\begin{align}
    \label{eq:xiestC1C2}
    (1+s)^{1/2}\norm{\nabla\xi(s)}_{\infty} + (1+s)\norm{\Delta \xi(s)}_{\infty} \leq C E_0
\end{align}
for $s \in [0,t]$.
Since $\xi = \Psi_\beta(\sigma)$ by definition, it also holds that $\sigma = \smash{\Psi_\beta^{-1}}(\xi)$.
Using the fact that $\norm{\sigma(t)}_{\infty} \leq \eta(t) \leq 1$, it follows that the estimates~\eqref{eq:xiestC9} and~\eqref{eq:xiestC1C2}
can be transferred to $\sigma$ using (uniform) smoothness of $\smash{\Psi_\beta^{-1}}$.
This results in
\begin{align}
    \label{eq:sigmafullest1}
    \norm{\sigma(s)}_{C^9} + (1+s)^{1/2} \norm{\nabla \sigma(s)}_\infty + (1+s) \norm{\Delta \sigma(s)}_\infty &\leq C E_0
\end{align}
for $s \in [0,t]$, which establishes~\eqref{eq:sigmaest}.

\emph{Step 2: Decomposing $v$.}
The Duhamel formulation of~\eqref{eq:pdev} reads
\begin{equation*}
    v(s) = \re^{sL}(u_0 - \phi) + \int_0^s \re^{(s - r)L} \widetilde{N}(r) \rd r
\end{equation*}
for $s \in [0,t]$, where we have abbreviated
\begin{equation}
    \label{eq:defNtilde}
    \widetilde{N}(r) \coloneq \Delta \sigma(r) (d_{\bot}\phi' - D\phi') + \abs{\nabla \sigma(r)}^2 (D \phi'' + \tfrac{1}{2}c\, \phi') + N(r) - g(r) \phi'.
\end{equation}
Applying the decomposition of Theorem~\ref{thm:decomposition} to both terms then shows that we have
\begin{equation}
    \label{eq:vdecomp}
    v(s) = \zeta(s) \phi' + w(s),
\end{equation}
with
\begin{subequations}
\label{eq:zetawdef}
\begin{align}
    \label{eq:defzeta}
    \zeta(s) &\coloneq \re^{sd_{\bot}\Delta_y} \langle e^*, u_0 - \phi\rangle_2+
    \int_0^s \re^{(s-r)d_{\bot}\Delta_y} \langle e^*,\widetilde{N}(r)\rangle_2 \rd r, \\
    \label{eq:winfbaseest}
    \norm{w(s)}_{\infty} &\leq C (1+s)^{-1}E_0 + C\int_0^s (1+(s-r))^{-1} \norm{\widetilde{N}(r)}_{\infty} \rd r
\end{align}
\end{subequations}
for $s \in [0,t]$. We will now bound $w$ and $\zeta$ separately.

\emph{Step 3a: Estimating $w$.}
We first notice that we can estimate
\begin{equation}
    \label{eq:Ntildeest}
    \norm{\widetilde{N}(s)}_{\infty} \leq C E_0 (1+s)^{-1}
\end{equation}
for $s \in [0,t]$, by applying the triangle inequality to~\eqref{eq:defNtilde} and then using~\eqref{eq:Nbaseest},~\eqref{eq:gest}, and~\eqref{eq:sigmafullest1}.
Substituting~\eqref{eq:Ntildeest} into~\eqref{eq:winfbaseest} and using Lemma~\ref{lem:convolution} then gives 
\begin{equation}
    \label{eq:west}
    \norm{w(s)}_{\infty} \leq C E_0 (1+s)^{-1}\log(2+s)
\end{equation}
for $s \in [0,t]$.

\emph{Step 3b: Estimating $\zeta$.}
We establish the estimate
\begin{align}
    \label{eq:zetaest}
    \norm{\zeta(s)}_{\infty} \leq C E_0 (1+s)^{-5/4}
\end{align}
for $s \in [0,t]$. 
For $s \leq 1$ this is seen from Lemma~\ref{lem:heat},~\eqref{eq:defzeta} and~\eqref{eq:Ntildeest}.
Hence, we now assume $s \geq 1$ and let $n = \lfloor s \rfloor \geq 1$.
We will exploit several cancellations induced by our tracking scheme.
For the first piece of cancellation, take the inner product of~\eqref{eq:defNtilde} with $e^*$ and use~\eqref{eq:normalization},~\eqref{eq:defdbot}, and Lemma~\ref{lem:ndispersion} to find
\begin{equation*}
    \langle e^*,\widetilde{N}(r)\rangle_2 =\langle e^*,N(r)\rangle_2 - g(r)
\end{equation*}
for $r \in [0,t]$. Substituting the above into~\eqref{eq:defzeta}, we thus obtain
\begin{equation}
    \label{eq:zeta1}
    \zeta(s) = \re^{sd_{\bot}\Delta_y} \langle e^*, u_0 - \phi\rangle_2 + \int_0^s \re^{(s-r)d_{\bot}\Delta_y}\langle e^*,N(r)\rangle_2 \rd r
    - \int_0^s \re^{(s-r)d_{\bot}\Delta_y} g(r) \rd r.
\end{equation}
Our choice of $g$ will ensure that the third term fully cancels the first term and partially cancels the second.
Indeed, using~\eqref{eq:defgnew} and the facts that $\varrho$ has unit integral and that $n = \lfloor s \rfloor \geq 1$, we find that the third term satisfies
\begin{equation}
\label{eq:T3a}
\begin{aligned}
    \int_0^s \re^{(s-r)d_{\bot}\Delta_y} g(r) \rd r
    &= \re^{s d_{\bot}\Delta_y}\int_0^s \varrho(r) \rd r \, \langle e^*,u_0 - \phi\rangle_2 \\
    &\qquad+ \,\sum_{k \in \bbN}\sqr[\bigg]{ \int_0^s \varrho(r-k) \rd r \int_{k-1}^k \re^{(s - r)d_{\bot}\Delta_y}\langle e^*,N(r)\rangle_2\rd r} \\
    &= \re^{s d_{\bot}\Delta_y}\langle e^*,u_0 - \phi \rangle_2 +\int_0^{n-1}\re^{(s-r)d_{\bot}\Delta_y}\langle e^*,N(r)\rangle_2 \rd r \\
    &\qquad+\,\int_0^s \varrho(r-n)\rd r \int_{n-1}^n \re^{(s-r)d_{\bot}\Delta_y}\langle e^*,N(r)\rangle_2 \rd r.
\end{aligned}
\end{equation}
Substituting~\eqref{eq:T3a} into~\eqref{eq:zeta1} yields
\begin{align*}
    \zeta(s) = \int_{n-1}^s \re^{(s-r)d_{\bot}\Delta_y} \langle e^*,N(r)\rangle_2 \rd r
    - \int_0^s \varrho(r - n)\rd r\int_{n-1}^{n} \re^{(s-r)d_{\bot}\Delta_y} \langle e^*,N(r)\rangle_2 \rd r.
\end{align*}
Applying Lemma~\ref{lem:heat} and the estimate~\eqref{eq:Nbaseest} to this ``fully canceled'' formula for $\zeta(s)$ then shows that~\eqref{eq:zetaest} also holds true for $s \in [0,t]$ with $s \geq 1$.

\emph{Step 4: Estimating $v$ and $N$.}
Combining~\eqref{eq:vdecomp} and~\eqref{eq:west}-\eqref{eq:zetaest} immediately yields
\begin{align}
    \label{eq:vestinfC1}
    \norm{v(s)}_{\infty} \leq C E_0 (1+s)^{-1}\log(2+s)
\end{align}
for $s \in [0,t]$, which is exactly~\eqref{eq:vkeyest}.
As a trivial consequence, we have $\norm{u(s)}_{\infty} \leq C$ for $s \in [0,t]$. By standard parabolic regularity theory~\cite{Lunardi_analytic_2013}, this implies $\smash{\chi(s)\norm{u(s)}_{C^9}} \leq C$ for $s \in [0,t]$, where we use that $\chi$ is supported on $[1/8,\infty)$.
The additional regularity of $u$ can then be transferred back to $v$ using~\eqref{eq:vurel} and~\eqref{eq:sigmafullest1}, which ultimately results in $\smash{\chi(s)\norm{v(s)}_{C^9}} \leq C$ for $s \in [0,t]$.
Interpolating this with~\eqref{eq:vestinfC1} then yields the rather crude bound
\begin{align}
    \label{eq:vestC2}
    \chi(s)\norm{v(s)}_{C^2} \leq C E_0^{3/4}(1+s)^{-3/4}
\end{align}
for $s \in [0,t]$, which is the final piece needed to obtain the desired estimate on the nonlinearity $N$.

\emph{Step 5: Conclusion.}
Substituting~\eqref{eq:gest},~\eqref{eq:sigmafullest1},~\eqref{eq:vestinfC1}, and~\eqref{eq:vestC2} into the nonlinear estimate~\eqref{eq:Nest1} obtained in Lemma~\ref{lem:Nestimate1} and counting powers of $E_0$ and $(1+s)$ we find
\begin{align*}
    \norm{N(s)}_{\infty} \lesssim E_0^{7/4}(1+s)^{-5/4}
\end{align*}
for $s \in [0,t]$. Combining this with~\eqref{eq:sigmafullest1} shows that~\eqref{eq:apkeyest} is satisfied.
Thus, it remains to show~\eqref{eq:gbetterest}.
To accomplish this, we combine the estimates~\eqref{eq:gest},~\eqref{eq:sigmafullest1}, and~\eqref{eq:vestinfC1} with the nonlinear bound~\eqref{eq:Nest2} from Lemma~\ref{lem:Nestimate2} to find
\begin{equation*}
    \norm{\re^{\frac{1}{8}d_{\bot}\Delta_y} \langle e^*,N(s)\rangle_2}_{\infty} \lesssim E_0^2 (1+s)^{-3/2} \log(2+s)
\end{equation*}
for $s \in [0,t]$. 
Inserting this into~\eqref{eq:defgnew}, while applying Lemma~\ref{lem:heat} and taking into account the support properties of $\varrho$, then yields~\eqref{eq:gbetterest}.

For the final claim, we insert $\sigma(0) \equiv g(0) \equiv 0$ and $v(0) = u_0 - \phi$ into~\eqref{eq:Nest1} and~\eqref{eq:template} to obtain
\begin{equation} \label{eq:eta0est}
\eta(0) \leq E_0 + C E_0^{-1}\norm{v(0)}^2_{\infty} \leq C E_0 \leq C E_0^{3/4},
\end{equation}
so long as $E_0 \leq 1$.

Now, let $C_1 > 0$ be an admissible constant for both~\eqref{eq:apkeyest} and~\eqref{eq:eta0est}, and set $\delta = \min\{\smash{(2C_1)^{-4/3}},1\}$.
Then, when $E_0 \leq \delta$, estimate~\eqref{eq:eta0est} implies $\eta(0) \leq 1/2$. In addition, estimate~\eqref{eq:apkeyest} shows that $\eta(t) \leq C_1 \smash{E_0^{3/4}} \leq 1/2$ for all $t \in [0,T]$ which satisfy $\eta(t) \leq 1$. Hence, it follows that $\eta(t) \leq 1$ for all $t \in [0,T]$ by continuity of $\eta$, as desired.
\end{proof}

\subsection{Construction of solutions to tracking scheme} \label{sec:extension}
Given a time $T>0$, and a solution $u$ to the reaction-diffusion system~\eqref{eq:rdefull} on $[0,T]$, we show that the a priori bounds established in Lemmas~\ref{lem:xiAP} and~\ref{lem:key} ensure the existence of a unique solution $(\sigma, v, g)$ to the tracking scheme defined by~\eqref{eq:pdesigma},~\eqref{eq:pdev}, and~\eqref{eq:defgnew} on the same interval $[0,T]$. In the proof of the following proposition, we obtain these solutions via a simple iterative extension procedure, which crucially relies on the following facts:
\begin{itemize}
    \item For $t\in[0,1]$ and $n\in\bbN$, the value of $g(n+t)$ is completely determined by $N(v(s),\sigma(s),g(s))$ on $[0,n]$.
    \item The solution $u$ and the forcing $g$ uniquely determine $\sigma$ and $v$.
\end{itemize}
\begin{proposition}[Extension argument]
    \label{prop:tracking}
There exists $\delta'> 0$ such that for any $T > 0$, $u_0 \in \Cub(\bbR^d)$ with $E_0 = \norm{u_0 - \phi}_\infty \leq \delta'$, and (classical) solution
\begin{align*}
 u \in C\big([0,T],\Cub(\bbR^d)\big) \cap C^\infty\big((0,T] \times \bbR^d\big)
\end{align*}
to~\eqref{eq:rdefull}, there exist unique smooth $\sigma,g \colon [0,T] \times \bbR^{d-1} \to \bbR$ and
\begin{align} \label{eq:vreg}
 v \in C\big([0,T],\Cub(\bbR^d)\big) \cap C^\infty\big((0,T] \times \bbR^d\big)
\end{align}
satisfying~\eqref{eq:pdesigma},~\eqref{eq:pdev}, and~\eqref{eq:defgnew} for $t \in [0,T]$.
\end{proposition}
\begin{proof}
Let us first assume that the conclusion holds with $T$ replaced by some $n \in \bbN_0$. We will show that with this assumption, the conclusion also holds with $T$ replaced by $t = n + t'$ for any $t' \in [0,1]$, as long as $n+t' \leq T$. The claim then follows by iteration (the case $n = 0$ is trivial).

Let $\sigma,v,g$ be the functions defined on $[0,n]$ which are supplied to us by the assumption. Using the support properties of $\varrho$, we see that there exists a unique smooth extension of $g$ to $[0,t]$ such that~\eqref{eq:defgnew} still holds. Taking $E_0 \leq \delta$, where $\delta > 0$ is as in Lemma~\ref{lem:key}, we have $\eta(s) \leq 1$ for $s \in [0,n]$.
Via~\eqref{eq:template}, this implies that $\norm{N(v(s),\sigma(s),g(s))}_{\infty} \leq E_0 (1+s)^{-5/4}$ 
holds for $s \in [0,n]$. Substituting this into~\eqref{eq:defgnew}, applying Lemma~\ref{lem:heat}, and recalling that $\varrho$ is supported on $[1/4,3/4]$, we obtain a $T$- and $u_0$-independent constant $C > 0$ such that
\begin{equation}
\label{eq:gestfinal}
\norm{g(s)}_{\infty} \leq C E_0 (1+s)^{-5/4}
\end{equation}
for $s \in [0,t]$. Hence, Lemma~\ref{lem:xiAP} yields a $T$- and $u_0$-independent constant $M_0 > 0$ such that, provided $E_0 \leq M_0$, there exists a unique smooth extension of $\sigma$ which solves~\eqref{eq:pdesigma} on $[0,t]$.
With the unique extensions of $\sigma$ and $g$ at hand, we use~\eqref{eq:vurel} to uniquely extend $v$.
This ensures that $v$ has regularity~\eqref{eq:vreg}, which justifies the calculations in Section~\ref{sec:modulation} which show that $v$ solves~\eqref{eq:pdev}. Thus, the result follows by taking $\delta' = \min\{M_0,\delta\}$.
\end{proof}

\section{Proofs of the main results} \label{sec:proof}

The proof of Theorem~\ref{thm:main} follows by combining the a priori bounds from Lemma~\ref{lem:key} with standard local existence and regularity theory for semilinear parabolic equations. At the same time, we prove Theorem~\ref{thm:HJ} by approximating the modulation function $\sigma$ with the solution $\tilde\sigma$ of the viscous Hamilton--Jacobi equation~\eqref{eq:HJ}. This is achieved by estimating the difference of their Cole--Hopf transforms, which satisfies a simple forced linear heat equation, allowing us to derive the required bounds via the Duhamel formulation.

\begin{proof}[Proof of Theorems~\ref{thm:main} and~\ref{thm:HJ}]
    Let $\delta,\delta' > 0$ be as in Lemma~\ref{lem:key} and Proposition~\ref{prop:tracking}, respectively. Set $\eps = \min\cur{\delta,\delta'}$ and assume $E_0 \leq \eps$.
    By standard theory~\cite{Lunardi_analytic_2013} for parabolic semilinear equations there exist a maximal time $T' \in (0,\infty]$ and a maximally defined (classical) solution
    \begin{align*}
     u \in C\big([0,T'),\Cub(\bbR^d)\big) \cap C^\infty\big((0,T') \times \bbR^d\big)
    \end{align*}
    of~\eqref{eq:rdefull} such that the following blow-up criterion
    \begin{equation}
        \label{eq:blowup}
        \limsup_{t \uparrow T'} \norm{u(t)}_{\infty} < \infty \implies T' = \infty 
    \end{equation}
    holds.
    Applying Proposition~\ref{prop:tracking} and using $E_0 \leq \delta'$, we thus obtain smooth $\sigma,g \colon [0,T') \times \bbR^{d-1} \to \bbR$ and
   \begin{align*}
    v \in C\big([0,T'),\Cub(\bbR^d)\big) \cap C^\infty\big((0,T') \times \bbR^d\big)
   \end{align*}satisfying~\eqref{eq:pdesigma},~\eqref{eq:pdev}, and~\eqref{eq:defgnew} for $t \in [0,T')$.
   By Lemma~\ref{lem:key} and the fact that $E_0 \leq \delta$ it then follows that~\eqref{eq:thmest4}-\eqref{eq:thmest1} and~\eqref{eq:gbetterest} hold for all $t \in [0,T')$. The uniform estimate for $u$ on $[0,T')$ provided by~\eqref{eq:thmest4} then implies that $T' = \infty$ using~\eqref{eq:blowup}. Thus,~\eqref{eq:thmest4}-\eqref{eq:thmest1} hold for all $t \geq 0$ and the proof of Theorem~\ref{thm:main} is complete.

   We proceed with the proof of Theorem~\ref{thm:HJ}. Throughout the proof, we denote by $C > 0$ any constant, which is independent of $t$ and $u_0$. Let $\tilde{\sigma} \in \smash{C\big([0,\infty),\Cub(\bbR^{d-1})\big)} \cap \smash{C^\infty\big((0,\infty) \times \bbR^{d-1}\big)}$ be the global (classical) solution to viscous Hamilton--Jacobi equation~\eqref{eq:HJ}, whose existence and uniqueness follow directly via the Cole--Hopf transform.
   Recalling the definitions of $\beta$~\eqref{eq:defbeta} and $\Psi_{\beta}$~\eqref{eq:CH}, we set
    \begin{align}
        \label{eq:defxixibarX}
        \xi \coloneq \Psi_\beta(\sigma), \qquad \tilde{\xi} \coloneq \Psi_\beta(\tilde{\sigma}), \qquad X \coloneq \tilde{\xi} - \xi.
    \end{align}
    so that $X$ satisfies
    \begin{equation}
    \label{eq:pdeX}
    \begin{aligned}
        \partial_t X &= d_{\bot}\Delta X + g(1+ \beta\xi), \\
        X(0) &= \Psi_\beta(\langle e^*,\phi - u_0\rangle_2).
    \end{aligned}
    \end{equation}
    Our strategy is to first show that $X$ remains small and then transfer this control to $\sigma - \tilde\sigma$. Specifically, we will establish
    \begin{equation}
        \label{eq:Xfullest}
        \norm{X(t)}_{\infty} \leq C \big(E_0^2 + E_0(1+t)^{-1}\big), \qquad t \geq 0.
    \end{equation}
    To this end, we begin by collecting several preliminary estimates.
    First, note that $\smash{\re^{-C \eps} \leq 1+ \beta \tilde{\xi}(0) \leq \re^{C \eps}}$. Using $\smash{\tilde{\xi}(t) = \re^{t d_\bot \Delta_y} \tilde{\xi}(0)}$ and positivity of the heat semigroup, this implies 
    $\smash{\re^{-C \eps} \leq 1+ \beta \tilde{\xi}(t) \leq \re^{C \eps}}$ for all $t \geq 0$. 
    This guarantees that the inverse Cole--Hopf transform of $\tilde{\xi}$ is well-defined and uniformly smooth.
    By contractivity of the heat semigroup, it also holds that $\smash{\norm{\tilde{\xi}(t)}_{\infty} \leq \norm{\tilde{\xi}(0)}_{\infty} \leq \ C E_0}$ for all $t \geq 0$.
    Therefore, applying the Cole--Hopf transform to $\sigma(t)$ and its inverse to $\smash{\tilde{\xi}(t)}$, while recalling estimate~\eqref{eq:thmest1}, it can be seen that
    \begin{align}
        \label{eq:sigmasigmatildebound}
        \norm{\sigma(t)}_{\infty}
        + \norm{\tilde{\sigma}(t)}_{\infty}
        + \norm{\xi(t)}_{\infty} \leq CE_0
    \end{align}
    for $t \geq 0$. Combining estimate~\eqref{eq:gbetterest} with~\eqref{eq:sigmasigmatildebound}, while recalling that $\varrho$ is supported on $[1/4,3/4]$, we obtain
    \begin{align}
        \label{eq:gxibestest}
        \norm{\xi(t)g(t)}_{\infty} &\leq C E_0^2 (1+t)^{-3/2} \log(2+t)
    \end{align}
    for $t \geq 0$. 
    As a final ingredient, an application of Taylor's theorem guarantees that
    \begin{align}
        \label{eq:ictaylor}
        \norm{X(0)}_{\infty} &\leq C E_0, \qquad        \norm{X(0) - \langle e^*,\phi - u_0 \rangle_2 }_{\infty} \leq CE_0^2.
    \end{align}
    Now we are in position to prove~\eqref{eq:Xfullest}.
    For $t \leq 1$, the claim follows directly by inserting~\eqref{eq:gbetterest}, \eqref{eq:gxibestest}, and~\eqref{eq:ictaylor} into the Duhamel representation of~\eqref{eq:pdeX} and applying Lemma~\ref{lem:heat}.
    Next, we treat the case $t \geq 1$.
    We first observe, since $\varrho$ is supported on $[1/4,3/4]$, that~\eqref{eq:defgnew} reduces to
    \begin{equation*}
        g(s) = \varrho(s)\re^{s d_{\bot}\Delta_y}\langle e^*,u_0 - \phi \rangle_2
    \end{equation*}
    for $s \in [0,1]$. Substituting this into the Duhamel formulation for~\eqref{eq:pdeX} and using that $\varrho$ has unit integral, yields the identity
    \begin{align*}
        X(1) &= \re^{d_{\bot}\Delta_y}\sqr[\Big]{X(0) -  \langle e^*,\phi - u_0 \rangle_2} 
        + \beta \int_0^1 \re^{(1 - s)d_{\bot}\Delta_y}\xi(s)g(s)\rd s.
    \end{align*}
    Hence, invoking~\eqref{eq:gxibestest} and~\eqref{eq:ictaylor}, and Lemma~\ref{lem:heat}, we infer that $\norm{X(1)}_{\infty} \leq C E_0^2$.
    Restarting the Duhamel formulation for~\eqref{eq:pdeX} at $t=1$, we also find the identity
    \begin{align*}
        X(t) = \re^{(t - 1)d_{\bot}\Delta_y}X(1) + \int_1^t \re^{(t - s)d_{\bot}\Delta_y}\sqr[\big]{g(s)(1+\beta\xi(s))}\rd s, \qquad t \geq 1.
    \end{align*}
    Combining this with our bound for $X(1)$, Lemma~\ref{lem:heat}, and estimates~\eqref{eq:gbetterest} and~\eqref{eq:gxibestest} and using that $\varrho$ is supported on $[1/4,3/4]$, it follows that~\eqref{eq:Xfullest} also holds for $t \geq 1$, as claimed.
    Finally, we invert the expressions in~\eqref{eq:defxixibarX} using~\eqref{eq:sigmasigmatildebound} and apply estimate~\eqref{eq:Xfullest} to arrive at
    \begin{align*}
        \norm{\sigma(t) - \tilde{\sigma}(t)}_{\infty} \leq C\norm{X(t)}_{\infty} \leq C \big(E_0^2 + E_0 (1+t)^{-1}\big)
    \end{align*}
    for $t \geq 0$. Combining this with~\eqref{eq:thmest4}-\eqref{eq:thmest1} and smoothness of $\phi$ shows that~\eqref{eq:HJeffest} holds, concluding the proof of Theorem~\ref{thm:HJ}.
\end{proof}

Subsequently, we prove Corollary~\ref{cor:oscillating} by combining the leading-order approximation of the front interface dynamics by the viscous Hamilton--Jacobi equation~\eqref{eq:HJ_pde} with the construction in~\cite[\S\,3]{matano_stability_2009} of small, infinitely oscillating solutions to~\eqref{eq:HJ_pde}.

\begin{proof}[Proof of Corollary~\ref{cor:oscillating}]
    Let $\varepsilon > 0$ be as in Theorem~\ref{thm:HJ}. It suffices to prove the claim in the case $d = 2$ (when $d \geq 3$, we can extend the solutions to be constant in the additional directions).
    
    \emph{Step 1: Construction.}
    Using the construction from~\cite[\S\,3]{matano_stability_2009}, we can find a smooth solution $\xi \colon [0,\infty) \times \bbR \to \bbR$ to the heat equation $\partial_t \xi = d_{\bot} \partial_{yy}\xi$,
    a sequence of times $t_n \in [0,\infty)$ diverging to infinity, and a sequence of points $y_n \in \bbR$ such that it holds
    \begin{align*}
    \abs{\xi(t,y)} \leq 2, 
    \qquad \xi(t_n,0) = (-1)^{n},
    \qquad \xi(t_n,y_n) = - (-1)^{n},
    \end{align*}
    for all $t \geq 0$, $y \in \bbR$, and $n \in \bbN$.
    We then define smooth functions
    \begin{align*}
        \tilde{\sigma}_{\delta} &\coloneq \Psi_\beta^{-1}(\delta \xi), \qquad 
        u_{0,\delta}(y,z) \coloneq \phi(z) - \tilde\sigma_{\delta}(0,y)\phi'(z),
    \end{align*}
    for $(y,z) \in \bbR^2$ and $\delta \in (0,(4\beta)^{-1})$.
    
    \emph{Step 2. Verification.}
    We now claim that the statement of the corollary is witnessed by the family of initial conditions $(u_{0,\delta})_{\delta \ll 1}$.
    To see this, first note that by setting $u_0 = u_{0,\delta}$ in~\eqref{eq:HJ_ic}, it follows using~\eqref{eq:normalization} that $\tilde{\sigma}_{\delta}$ is exactly the global  solution to~\eqref{eq:HJ}.
    Moreover, from the properties of $\xi$ and~\eqref{eq:CHinv} we also establish
    \begin{subequations}
    \begin{align}
    \label{eq:u0deltaic}
    E_{0,\delta} \coloneq \norm{u_{0,\delta} - \phi}_\infty \lesssim \norm{\tilde\sigma_{\delta}(0)}_\infty &\lesssim \delta,\\
    \abs{\tilde{\sigma}_{\delta}(t_n,0) - (-1)^n \delta}+\abs{\tilde{\sigma}_{\delta}(t_n,y_n) + (-1)^n \delta} &\lesssim \delta^2,
    \end{align}
    \end{subequations}
    for all $\delta \in (0,(4\beta)^{-1})$ and $n \in \bbN$.
    Now let $u_{\delta}$ denote the solution to~\eqref{eq:rde} with initial condition $u_{0,\delta}$.
    Using the smoothness of $\phi$, an application of Theorem~\ref{thm:HJ} then yields that
    \begin{align*}
        \limsup_{n \to \infty} \sup_{z \in \bbR^d} \abs{u_{\delta}(t_n,0,z) - \phi(z - c t_n - (-1)^{n}\delta)} \lesssim \delta^2,\\
        \limsup_{n \to \infty} \sup_{z \in \bbR^d} \abs{u_{\delta}(t_n,y_n,z) - \phi(z - c t_n + (-1)^{n}\delta)} \lesssim \delta^2,
    \end{align*}
    for all $\delta \in (0,\min\cur{(4\beta)^{-1},\eps})$.
    Since $\phi' \not\equiv 0$ by Hypothesis~\ref{hyp:spectrum_1d}-\ref{it:spectrum_normal}, the mean value theorem finally yields constants $C_0,c_0 > 0$ such that it holds:
    \begin{align*}
        \limsup_{n \to \infty} \, \inf_{a \in \bbR}
        \, \sup_{(y,z) \in \bbR^d} \, \abs{u_{\delta}(t_n,y,z) - \phi(z - a)} \geq 2 c_0\, \delta - C_0\, \delta^2 \geq c_0\,\delta > 0
    \end{align*}
    for all $\delta \in (0,\min\cur{(4\beta)^{-1},\eps,c_0 C_0^{-1}})$, so that~\eqref{eq:infosc} holds.
    Since $E_{0,\delta}$ can be arbitrarily small due to~\eqref{eq:u0deltaic}, this completes the proof.
\end{proof}

Finally, the proof of the asymptotic stability result in Theorem~\ref{thm:decay} is based on an iterative application of the following lemma, which shows that transversely localized perturbations of the front remain transversely localized and that their $L^\infty$-norm is reduced by a factor of one half after a finite time. Transverse spatial localization is propagated in time via a standard Gr\"onwall argument, while the decay in norm follows by combining the estimate in Theorem~\ref{thm:HJ} with the observation that sufficiently localized solutions to the viscous Hamilton--Jacobi equation~\eqref{eq:HJ_pde} converge uniformly to zero.

\begin{lemma}
    \label{lem:decayiterate}
    There exists $\eps > 0$ such that, whenever  $u_0 \in \Cub(\bbR^d)$ satisfies~\eqref{eq:E0ineq} and
    \begin{equation*}
		\lim_{\abs{y} \to \infty} \sup_{z \in \bbR}\, \abs{u_0(y,z) - \phi(z)} = 0,
	\end{equation*}
    the following statements hold:
    \begin{itemize}
        \item There exists a unique global solution $u$ to~\eqref{eq:rdefull} with regularity~\eqref{eq:soluglobal}.
        \item The solution $u$ retains its transverse localization: it holds
        \begin{equation}
            \label{eq:preservedecay}
            \lim_{\abs{y} \to \infty} \, \sup_{z \in \bbR}\, \abs{u(t,y,z) - \phi(z - ct)} = 0, \qquad t \geq 0.
        \end{equation}
        \item There exists a time $t_1 > 0$ such that
        \begin{equation}
            \label{eq:decayconvergence}
            \sup_{(y,z) \in \bbR^d} \, \abs{u(t,y,z) - \phi(z - ct)} \leq \tfrac{1}{2} E_0, \qquad t \geq t_1.
        \end{equation}
    \end{itemize}
\end{lemma}
\begin{proof} Throughout the proof, we denote by $C > 0$ any constant, which is independent of $t$ and $u_0$.

Taking $\eps > 0$ sufficiently small, let $u$ and $\tilde{\sigma}$ be the global solutions to~\eqref{eq:rdefull} and~\eqref{eq:HJ} with regularity~\eqref{eq:soluglobal} and~\eqref{eq:sigmatildereg} obtained from Theorem~\ref{thm:HJ},
so that the estimate~\eqref{eq:HJeffest} holds.
Recalling that $u_{\mathrm{tf}}(t,y,z) = \phi(z - ct)$ was introduced in~\eqref{eq:travelingfront}, we set $w \coloneq u - u_{\mathrm{tf}}$, so that $w$ has regularity~\eqref{eq:soluglobal} and is a solution to
    \begin{equation}
    \label{eq:pdevtildeextra}
    \begin{aligned}
        \partial_t w &= D \Delta w + f(u_{\mathrm{tf}} + w) - f(u_{\mathrm{tf}}), \\
        w(0) &= u_0 - \phi.
    \end{aligned}
    \end{equation}
    For $h \in \Cub(\bbR^d)$, we define the sublinear functional
\begin{equation*}
    M\sqr{h} = \limsup_{R \to \infty} \sup_{\abs{y} \geq R,\, z \in \bbR} \abs{h(y,z)}.
\end{equation*}
We claim that
\begin{align*}
    M\sqr{\re^{t D \Delta}h} \lesssim M\sqr{h}
\end{align*}
for all $t \geq 0$ and $h \in \Cub(\bbR^d)$. Diagonalizing the symmetric matrix $D$ as $D = JD_0J^{-1}$ with $D_0 = \mathrm{diag}(d_1,\ldots,d_n) \in \bbR^{n \times n}$ a positive diagonal matrix and $J \in \bbR^{n \times n}$ invertible, this bound indeed follows from the Gaussian decay in space of the associated temporal Green's function $G(t,x) = J\,\mathrm{diag}(G_1(t,x),\ldots,G_n(t,x))\,J^{-1}$ with 
\begin{align*}
G_i(t,x) = \frac{\re^{-\frac{\abs{x}^2}{4d_i t}}}{(4\pi d_i t)^{d/2}}, \qquad i = 1,\ldots,n.
\end{align*}
    Hence, from the Duhamel formulation of~\eqref{eq:pdevtildeextra} we can derive the estimate
    \begin{align*}
        M\sqr{w(t)} &\leq C M\sqr{w(0)} + C\int_0^t M\sqr{f(u_{\mathrm{tf}}(s) + w(s)) - f(u_{\mathrm{tf}}(s))} \rd s \\
        &\leq CM\sqr{w(0)} + C\int_0^t M\sqr{w(s)} \rd s
    \end{align*}
    for $t \geq 0$. Since we have $M\sqr{w(0)} = 0$ by the assumption on $u_0$, and the map $t \mapsto M\sqr{w(t)}$ is continuous, we conclude by Gr\"onwall's lemma that $M\sqr{w(t)} = 0$ for all $t \geq 0$, which yields~\eqref{eq:preservedecay}.

    Next, note that $\lim_{\abs{y} \to \infty}\tilde{\sigma}(0,y) = 0$ by the assumption on $u_0$ and~\eqref{eq:HJ_ic}. Hence, using that solutions with localized initial data to the heat equation decay uniformly to $0$ as $t \to \infty$, it follows, after applying the Cole--Hopf transform, that $\lim_{t \to \infty} \norm{\tilde{\sigma}(t)}_{\infty} = 0$.
    On the other hand, we deduce from~\eqref{eq:HJeffest} that
    \begin{equation*}
        \sup_{(y,z) \in \bbR^d} \, \abs{u(t,y,z) - \phi(z - ct)} \leq C\big( E_0^2 + E_0 (1+t)^{-1}\log(2+t) + \norm{\tilde{\sigma}(t)}_{\infty}\big)
    \end{equation*}
    for $t \geq 0$. 
    Thus, after additionally imposing $\eps \leq (4C)^{-1}$, we can find $t_1 > 0$, depending only on $u_0$, such that~\eqref{eq:decayconvergence} holds.
\end{proof}

\begin{proof}[Proof of Theorem~\ref{thm:decay}]
Using translational symmetry of~\eqref{eq:rde}, we iterate Lemma~\ref{lem:decayiterate} to find an increasing sequence of times $t_n \in [0,\infty)$ such that
\begin{equation*}
    \sup_{(y,z) \in \bbR^d} \, \abs{u(t,y,z) - \phi(z - ct)} \leq E_0 2^{-n}, \qquad t \geq t_n,
\end{equation*}
which proves the claim.
\end{proof}

\section{Outlook and open problems}
The nonlinear stability framework developed in this paper extends beyond the class of semilinear reaction-diffusion systems considered here. Provided that the linearization generates a $C^0$-semigroup on $\Cub(\bbR^d)$ with uniformly damped high-frequency component, we expect the approach to apply to bistable fronts in general semilinear dissipative systems and even to quasilinear dissipative problems, provided sufficient regularity control is available within the nonlinear iteration scheme. Such $L^\infty$-based regularity control may be obtained via energy estimates in uniformly local Sobolev spaces; see~\cite{Alexopoulos_nonlinear_2026}.
 
The fact that only derivatives of the modulation $\sigma$ enter the nonlinearity of the perturbation equation~\eqref{eq:pdev} suggests that one may allow for \emph{modulational initial data} of the form
\begin{align} \label{eq:moddata} 
u_0(y,z) = \phi(z - \sigma_0(y)) + v_0(y,z),
\end{align}
where $\sigma_0 \in C_{\mathrm{ub}}^1(\bbR^{d-1})$ and $v_0 \in \Cub(\bbR^d)$ satisfy $\|\nabla \sigma_0\|_\infty + \|v_0\|_\infty \ll 1$. Since $\|\sigma_0\|_\infty$ need not be small, such initial data are not necessarily close to any fixed translate of the planar front. Proving that the associated solution $u$ to~\eqref{eq:rdefull} stays close to a modulated planar front would therefore constitute a global stability result. Similar \emph{modulational stability} results~\cite{Iyer_mixing_2019,alexopoulos_nonlinear_2025} have been established for wave trains in one spatial dimension, and we expect that ideas can be adapted to the present setting. 

The gradient bound $\smash{\|\nabla \re^{t d_\perp \Delta_y} \sigma\|_\infty \lesssim t^{-1/2} \|\sigma\|_{\mathrm{BMO}}}$ suggests a further relaxation to initial modulations $\sigma_0$ in~\eqref{eq:moddata} of bounded mean oscillation. In particular, this permits choices such as $\sigma_0(y) = \log(1+\delta |y|)$ for $0<\delta\ll 1$, for which $\|\nabla \sigma_0\|_\infty\ll 1$ and $|\sigma_0(y)|\to\infty$ as $|y|\to\infty$, so that $\phi(z - \sigma_0(y))$ represents a \emph{curved front}. Modulational stability for slowly linearly growing initial modulations $\sigma_0$ with $\sigma_0'(y) = \alpha\delta \tanh(\delta y)$ and $0 < \delta \ll \alpha$, has been obtained in two-dimensional reaction-diffusion systems under exponentially localized perturbations~\cite{Haragus_corner_2006}. It is an open question whether stability can also be established with respect to (partly) nonlocalized perturbations.

Existence and stability of bistable curved fronts that are not almost planar, i.e., which cannot be regarded as modulated fronts $\phi(z - \sigma_0(y))$ with $\|\nabla \sigma_0\|_\infty$ small, have been obtained in scalar reaction-diffusion equations and in specific multicomponent reaction-diffusion systems obeying a comparison principle; see, for instance,~\cite{Hamel_existence_2005,Roquejoffre_nontrivial_2009,Ninomiya_existence_2005,Wang_traveling_2012} and numerous references therein. To the best of the authors' knowledge, the nonlinear stability of not-almost-planar bistable curved fronts in general multicomponent reaction-diffusion systems remains an open problem, both against localized and against nonlocalized perturbations.

Another possibility is to augment~\eqref{eq:rde} with a random forcing term $\zeta$, as was done in~\cite{vandenbosch_local_2025}.
It was already observed in \cite{sendina-nadal_wave_1998} that in this case the front motion is governed by a viscous Hamilton--Jacobi equation with an additional random forcing term $\smash{\tilde{\zeta}}$ derived from $\zeta$, so that
\begin{equation}
    \label{eq:KPZ}
    \partial_t \sigma = d_{\bot}\Delta \sigma + \tfrac{1}{2}\abs{\nabla \sigma}^2 + \tilde{\zeta}.
\end{equation}
When $\smash{\tilde{\zeta}}$ is space-time white noise, \eqref{eq:KPZ} is exactly the Kardar--Parisi--Zhang (KPZ) model for random asymmetric surface growth~\cite{kardar_dynamic_1986}.
An interesting feature of KPZ is that it is only well-posed after a renormalization which is (formally) achieved by adding `$-\infty$' to the right-hand side of~\eqref{eq:KPZ}; see e.g.~\cite{hairer_solving_2013}.
In our context, this suggests that the average speed of the front will diverge as the correlation length of $\smash{\tilde{\zeta}}$ approaches zero, showing that rough noise can have an outsized effect on front propagation.
A proof of this fact would be of great interest and would require techniques from the study of singular stochastic PDEs.
Still, we expect that our methods to treat~\eqref{eq:rde} can be useful to future researchers attempting to prove such a result.

\appendix

\section{Uniform semigroup bounds}
The following lemma is used in Section~\ref{subsec:semigroupbounds} to pass from $\xi$-dependent bounds on $\re^{t\hat{L}_{\xi}}$ to $\xi$-independent bounds.
We expect that the following argument is known, but were unable to locate a suitable statement in the literature.

\begin{lemma}
    \label{lem:uniformsemigroup}
    Let $X$ be a Banach space, let $A$ be the generator of a $C_0$-semigroup $(\re^{tA})_{t \geq 0}$ on $X$, and let $S \subset \mathcal{B}(X)$ be compact with respect to the uniform operator topology.
    Suppose that for every $B \in S$, there exists $M_B > 0$ such that the inequality
    \begin{align}
        \label{eq:nonuniformsemigroup}
        \norm{\re^{t(A + B)}}_{\mathcal{B}(X)} \leq M_B,
    \end{align}
    holds for all $t \geq 0$.
    Then, for every $\eps > 0$, there exists $K_{\eps} > 0$ such that the inequality
    \begin{align*}
        \norm{\re^{t(A + B)}}_{\mathcal{B}(X)} \leq K_{\eps} \re^{\eps t}
    \end{align*}
    holds for all $t \geq 0$ and $B \in S$.
\end{lemma}
\begin{proof}
    Fix some $B \in S$.
    By using~\eqref{eq:nonuniformsemigroup} and applying the bounded perturbation theorem, cf.~\cite[Theorem III.1.3]{engel_oneparameter_2000}, it follows that there exists an open neighborhood $U_{B} \subset \mathcal{B}(X)$ of $B$ such that the inequality
    \begin{align}
        \norm{\re^{t(A + B')}}_{\mathcal{B}(X)} \leq M_B \re^{\eps t}
    \end{align}
    holds for all $t \geq 0$ and $B' \in U_{B}$.
    The family $(U_{B})_{B \in S}$ forms an open cover of $S$.
    So, by compactness there exists a subcover $(U_{B})_{B \in T}$ with $T \subset S$ finite.
    Upon taking $K_{\eps} = \max_{B \in T} M_B$, the result follows.
\end{proof}
\begin{remark}
    The exponential term $\re^{\eps t}$ cannot be omitted from Lemma~\ref{lem:uniformsemigroup}.
    This can be seen by considering the following example:
    \begin{equation*}
    X = \bbR^2,\qquad A = 0, \qquad
    S = \cur[\bigg]{\begin{pmatrix}
    -a^2 & a \\ 0 & -a^2
\end{pmatrix}}_{a \in [0,1]}.
\end{equation*}
\end{remark}

\printbibliography

\end{document}